\documentclass[12pt,a4paper]{amsart}
\usepackage{amsmath,amsfonts,latexsym,amssymb}
\usepackage[active]{srcltx} 

\newtheorem{theorem}{Theorem}[section]

\newtheorem{corollary}[theorem]{Corollary}
\newtheorem{lemma}[theorem]{Lemma}
\newtheorem{proposition}[theorem]{Proposition}
\newtheorem{remark}[theorem]{Remark}

\newcommand{\N}{\mathbb N}
\newcommand{\Z}{\mathbb Z}
\newcommand{\Q}{\mathbb Q}
\newcommand{\F}{\mathbb F}
\newcommand{\A}{\mathbb A}
\newcommand{\Pp}{\mathbb P}

\newcommand{\K}{{\sf K}}
\newcommand{\M}{{\sf M}}
\newcommand{\E}{{\sf E}}

\newcommand{\fq}{\F_{\hskip-0.7mm q}}
\newcommand{\cfq}{\overline{\F}_{\hskip-0.7mm q}}

\def\ifm#1#2{\relax \ifmmode#1\else#2\fi}

\newcommand{\klk}    {\ifm {,\ldots,} {$,\ldots,$}}

\newcommand{\plp}    {\ifm {+\cdots+} {$+\ldots+$}}

\newcommand{\om}[2]   {{#1}_0 \klk {#1}_{#2}}

\newcommand{\xo}[1]  {\ifm {\om X {#1}} {$\om X {#1}$}}
\newcommand{\xon}    {\ifm {\om X n} {$\om X n$}}

\begin{document}

\title[Rational points of singular complete intersections]{Polar varieties, Bertini's theorems and
number of points of singular complete intersections over a finite
field}

\author[A. Cafure et al.]{Antonio Cafure${}^{1,2,3}$,
        Guillermo Matera${}^{2,3}$ and
        Melina Privitelli${}^{3,4}$
}

\address{${}^{1}$Ciclo B\'asico Com\'un,
                Universidad de Buenos Aires,
                Ciudad Universitaria, Pabell\'on III
                (1428) Buenos Aires, Argentina}
\address{${}^{2}$Instituto del Desarrollo Humano,
                Universidad Nacional de Gene\-ral Sarmiento, J.M. Guti\'errez 1150
                (B1613GSX) Los Polvorines, Buenos Aires, Argentina}
\email{\{acafure,\,gmatera\}@ungs.edu.ar}
\address{${}^{3}$National Council of Science and Technology (CONICET),
                Ar\-gentina}
\email{mprivitelli@conicet.gov.ar}
\address{${}^{4}$Instituto de Ciencias,
                Universidad Nacional de Gene\-ral Sarmiento, J.M. Guti\'errez 1150
                (B1613GSX) Los Polvorines, Buenos Aires, Argentina}

\thanks{Research was partially supported by
grant PIP CONICET 11220090100421. {\it Corresponding Author}.
Guillermo Matera, Instituto del Desarrollo Humano, Universidad
Nacional de Gene\-ral Sarmiento, J. M. Guti\'errez 1150 (B1613GSX)
Los Polvorines, Buenos Aires, Argentina. E-mail: {\tt
gmatera@ungs.edu.ar}.}

\date{\today}
\maketitle

\begin{abstract}
Let $V\subset\Pp^n(\cfq)$ be a complete intersection defined over a
finite field $\fq$ of dimension $r$ and singular locus of dimension
at most $s$, and let $\pi:V\to\Pp^{s+1}(\cfq)$ be a ``generic''
linear mapping. We obtain an effective version of the Bertini
smoothness theorem concerning $\pi$, namely an explicit upper bound
of the degree of a proper Zariski closed subset of $\Pp^{s+1}(\cfq)$
which contains all the points defining singular fibers of $\pi$. For
this purpose we make essential use of the concept of polar variety
associated to the set of exceptional points of $\pi$. As a
consequence of our effective Bertini theorem we obtain results of
existence of smooth rational points of $V$, namely conditions on $q$
which imply that $V$ has a smooth $\fq$--rational point. Finally,
for $s=r-2$ and $s=r-3$ we obtain estimates on the number of
$\fq$--rational points and smooth $\fq$--rational points of $V$, and
we discuss how these estimates can be used in order to determine the
average value set of ``small'' families of univariate polynomials
with coefficients in $\fq$.
\end{abstract}
%
%
\section{Introduction}
%
%
Let $\fq$ be the finite field of $q$ elements and let $\cfq$ be the
algebraic closure of $\fq$. We denote by $\Pp^n(\fq)$,
$\Pp^n:=\Pp^n(\cfq)$, $\A^n(\fq)$ and $\A^n:=\A^n(\cfq)$ the
$n$--dimensional projective and affine spaces defined over $\fq$ and
$\cfq$ respectively. For any affine or projective variety $V$
defined over $\fq$, we denote by $V(\fq)$ the set of $\fq$--rational
points of $V$, namely the set of points of $V$ with coordinates in
$\fq$, and by $|V(\fq)|$ its cardinality. Observe that, for any
$r\ge 0$, we have
$$p_r:=|\Pp^r(\fq)|= q^r + \cdots + q + 1.$$

Let $V\subset\Pp^n$ be an ideal--theoretic complete intersection
defined over $\fq$, of dimension $r$, multidegree
$\boldsymbol{d}:=(d_1\klk d_{n-r})$ and singular locus of dimension
$s\ge 0$. The main results of this paper are estimates on $|V(\fq)|$
and conditions on $q$ which imply that $V(\fq)$ is not empty. All
these estimates and conditions will be expressed in terms of $r$,
$\boldsymbol{d}$ and $s$.

In a fundamental work \cite{Deligne74}, P. Deligne has shown that if
$V$ is nonsingular, then
\begin{equation}\label{eq: estimate deligne intro}
\big||V(\fq)|-p_r\big|\le b_r'(n,\boldsymbol{d})\,q^{r/2},
\end{equation}
where $b_r'(n,\boldsymbol{d})$ is the $r$th primitive Betti number
of any nonsingular complete intersection of $\Pp^n$ of dimension $r$
and multidegree $\boldsymbol{d}$ (see, e.g., \cite[Theorem
4.1]{GhLa02a} for an explicit expression of $b_r'(n,\boldsymbol{d})$
in terms of $n$, $r$ and $\boldsymbol{d}$).

This result has been extended by C. Hooley and N. Katz to singular
complete intersections. More precisely, in \cite{Hooley91} it is
proved that if the singular locus of $V$ has dimension at most $s\ge
0$, then
\begin{equation}\label{eq: estimate hooley-katz}
|V(\fq)|=p_r+ \mathcal{O}(q^{{(r+s+1)}/{2}}),
\end{equation}
where the constant implied by the $\mathcal{O}$--notation depends
only on $n$, $r$ and $\boldsymbol{d}$, and it is not explicitly
given.

In \cite{GhLa02a} (see also \cite{GhLa02}), S. Ghorpade and G.
Lachaud have obtained the following explicit version of (\ref{eq:
estimate hooley-katz}):
\begin{equation}\label{eq: estimate GL intro}
\big||V(\fq)|-p_r\big|\le
b_{r-s-1}'(n-s-1,\boldsymbol{d})\,q^{{(r+s+1)}/{2}}+
C(n,r,\boldsymbol{d})q^{{(r+s)}/{2}},
\end{equation}
where $C(n,r,\boldsymbol{d}):=9\cdot
2^{n-r}\big((n-r)d+3\big)^{n+1}$ and $d:=\max_{1\le i\le n-r}d_i$.

From the point of view of possible applications of (\ref{eq:
estimate GL intro}), the fact that $C(n,r,\boldsymbol{d})$ depends
exponentially on $n$ may be an inconvenience. This is particularly
the case if $V$ is a hypersurface, because $C(n,r,\boldsymbol{d})$
becomes exponential in the degree of $V$. For this reason, in
\cite{CaMa07} it is shown that, if $V$ is normal, then one has
\begin{equation}\label{eq: estimate CM intro}
\big||V(\fq)|-p_r\big|\le b_1'(n-r+1,\boldsymbol{d})\,q^{r-1/2}+ 2
\big((n-r)d\delta\big)^2q^{r-1},
\end{equation}
provided that $q>2(n-r)d\delta+1$ holds, where $\delta=d_1\cdots
d_{n-r}$ is the degree of $V$. This solves the exponential
dependency on $n$ of the error term in (\ref{eq: estimate GL intro})
for $s=r-2$ and $q$ large enough.
%
%
\subsection{Our contributions}
A fundamental tool for our work is an effective version of the
Bertini smoothness theorem. The Bertini smoothness theorem asserts
that a generic $(r-s-1)$--dimensional linear section of a complete
intersection $V\subset\Pp^n$ of dimension $r$ and singular locus of
dimension $s$ is nonsingular. With notations as above, an effective
version of this result establishes a threshold
$C(n,r,s,\boldsymbol{d})$ such that for $q>C(n,r,s,\boldsymbol{d})$
there exists a nonsingular unidimensional linear section defined
over $\fq$ of a complete intersection defined over $\fq$. In this
paper we show the following theorem (see Theorem \ref{theorem:
Bertini} and Corollary \ref{coro: existencia fibra no singular
fq-definible} below).
\begin{theorem}\label{theorem: Bertini intro}
Let $V\subset\Pp^n$ be a complete intersection defined over $\fq$,
of dimension $r$, degree $\delta$, multidegree
$\boldsymbol{d}:=(d_1\klk d_{n-r})$ and singular locus of dimension
at most $s\ge 0$. Let $D:=\sum_{i=1}^{n-r} (d_i-1)$. Then for
$q>(n+1)^2D^{r-s-1}\delta$ there exist nonsingular
$(r-s-1)$--dimensional linear sections of $V$ defined over $\fq$.
\end{theorem}
%
%
We remark that \cite{Ballico03} and \cite{CaMa07} provide effective
versions of the Bertini smoothness theorem for hypersurfaces and
normal complete intersections respectively. Theorem \ref{theorem:
Bertini intro} significantly improves and generalizes both results.

The linear sections underlying Theorem \ref{theorem: Bertini intro}
are obtained as (the Zariski closure of) fibers of a ``generic''
linear mapping $\pi:V\to\Pp^{s+1}$. For this purpose, it is
necessary to analyze the set $S$ of critical points of $\pi$. Our
treatment of the set $S$ relies on the notion of polar varieties.
Polar varieties are a classical concept of projective geometry
which, in its modern formulation, was introduced in the 1930's by F.
Severi and J. Todd. Around 1975 a renewal of the theory of polar
varieties took place with essential contributions due to R. Piene
\cite{Piene78}, B. Teissier \cite{Teissier82} and others (see
\cite{Teissier88} for a historical account and references). Our main
result in connection with polar varieties is a genericity condition
on $\pi$ which implies that the polar variety associated to the
exceptional locus of $\pi$ has the expected dimension (Theorem
\ref{theorem: dimension polar variety intrinseca}).

More precisely, let $\boldsymbol{\lambda}\in(\Pp^n)^{s+2}$ denote
the matrix of coefficients of the linear forms defining $\pi$. We
show that there exists a hypersurface of $(\Pp^n)^{s+2}$ which
contains all the points $\boldsymbol{\lambda}$ for which the
exceptional locus of $\pi$ has not the expected dimension. In order
to bound the degree of this hypersurface, we use tools from
intersection theory for products of projective spaces, such as a
multiprojective version of the B\'ezout theorem (see, e.g.,
\cite[Theorem 1.11]{DaKrSo13}). Combining this with the results on
the number of $\fq$--rational points of multiprojective
hypersurfaces of Section \ref{section: zeros multih hypersurfaces}
we obtain suitable bounds on the number of nonsingular linear
sections of $V$ defined over $\fq$.

Next we obtain conditions on $q$ which imply that the variety $V$
under consideration has a smooth $\fq$--rational point. A classical
problem is that of establishing conditions which imply that a given
variety has an $\fq$--rational point. Nevertheless, in several
number--theoretical applications is not just an $\fq$--rational
point what is required, but a {\em smooth} $\fq$--rational point
(see, e.g., \cite{LeSc73}, \cite{Wooley08}, \cite{Zahid10}).

A standard approach to this question consists of combining a lower
bound for the number of $\fq$--rational points with an upper bound
for the number of singular $\fq$--rational points of $V$. Instead of
doing this, we use our effective Bertini theorem, namely we obtain a
condition on $q$ which implies that there exists a nonsingular
$(r-s-1)$--dimensional linear section $S$ of $V$ defined over $\fq$,
and apply Deligne's estimate (\ref{eq: estimate deligne intro}) to
this linear section. As the linear section $S$ is contained in the
smooth locus $V_{\rm sm}:=V\setminus{\rm Sing}(V)$, the existence of
an $\fq$--rational point of $S$ implies that of a smooth
$\fq$--rational point of $V$. More precisely, we obtain the
following result (see Corollaries \ref{coro: existence codim 2} and
\ref{coro: existence codim 3}).
\begin{theorem}
Let $V\subset\Pp^n$ be a complete intersection defined over $\fq$,
of dimension $r$, degree $\delta$, multidegree $\boldsymbol{d}$ and
singular locus of dimension at most $s$. Let $D:=\sum_{i=1}^{n-r}
(d_i-1)$. If either $s=r-2$ and $q>2(D+2)^2\delta^2$, or $s=r-3$ and
$q>3D(D+2)^2\delta$ holds, then $V$ has a smooth $\fq$--rational
point.
\end{theorem}

Finally, we obtain estimates on the number of $\fq$--rational points
and smooth $\fq$--rational points of a complete intersection for
which the singular locus has dimension $s$ at most $r-2$ or $r-3$.
For this purpose, assuming that there exists a linear mapping
$\pi:V\to\Pp^{s+1}$ defined over $\fq$ which is generic in the sense
above, we express $V$ as the union of $p_{s+1}:=|\Pp^{s+1}(\fq)|$
linear sections of $V$ of dimension $r-s-1$, namely the Zariski
closure of the fibers of the points of $\Pp^{s+1}(\fq)$ under $\pi$.
``Most'' fibers will be nonsingular and thus Deligne's estimate can
be applied to them, while the $\fq$--rational points lying in the
remaining fibers do not make a significant contribution to the
estimate. Summarizing, we obtain the following result (see
Corollaries \ref{coro: estimate normal variety} and \ref{coro:
estimate codim 3} below).
\begin{theorem}\label{theorem: estimate intro}
Let $V\subset\Pp^n$ be a complete intersection defined over $\fq$,
of dimension $r$, degree $\delta$, multidegree $\boldsymbol{d}$ and
singular locus of dimension at most $s\in\{r-3,r-2\}$. Let
$D:=\sum_{i=1}^{n-r} (d_i-1)$. Then, for $s\le r-2$, we have:
$$\begin{array}{rl}\big||V(\fq)|-p_r\big|&\le
(\delta(D-2)+2)q^{r-1/2} +14D^2\delta^2q^{r-1},\\
\big||V_{\rm sm}(\fq)|-p_r\big|&\le (\delta(D-2)+2)q^{r-1/2}
+8(r+1)D^2\delta^2q^{r-1}.\end{array}$$
On the other hand, for $s\le r-3$, we have:
$$\begin{array}{rl}\big||V(\fq)|-p_r\big|&\le
14D^3\delta^2q^{r-1},\\
\big||V_{\rm sm}(\fq)|-p_r\big|&\le
(34r-20)D^3\delta^2q^{r-1}.\end{array} $$
\end{theorem}

Our estimates for the number of $\fq$--rational points follow the
pattern of (\ref{eq: estimate GL intro}) for $s=r-2$ or $s=r-3$, but
differ from (\ref{eq: estimate GL intro}) in that the exponential
dependency on $n$ is not present. In Section
\ref{section:estimacion} we show that Theorem \ref{theorem: estimate
intro} yields a more accurate estimate than (\ref{eq: estimate GL
intro}) in the case $s=r-2$ and $s=r-3$ for varieties of large
dimension, say $r\ge (n+1)/2$, or small degree, say $\delta\le
(2(n-r))^{n-r}$. On the other hand, (\ref{eq: estimate GL intro})
may be preferable to Theorem \ref{theorem: estimate intro} for
varieties of small dimension and large degree. In this sense, we may
say that Theorem \ref{theorem: estimate intro} complements (\ref{eq:
estimate GL intro}) for $s=r-2$ and $s=r-3$. Finally, Theorem
\ref{theorem: estimate intro} improves (\ref{eq: estimate CM intro})
for normal varieties in that it holds without any restriction on
$q$, while the latter only holds for $q>2(n-r)d\delta+1$.

We end this paper discussing a problem that requires the estimates
underlying Theorem \ref{theorem: estimate intro}: the average value
set of ``small'' families of polynomials. We sketch how we apply
Theorem \ref{theorem: estimate intro} in order to determine the
average value set of families of univariate polynomials of $\fq[T]$
of fixed degree where certain coefficients are fixed. Since this
problem is concerned with a complete intersection of ``low'' degree,
our estimate yields a significant gain compared with what is
obtained by means of (\ref{eq: estimate GL intro}).
%
%
\section{Notions, notations and preliminary results}
We use standard notions and notations of commutative algebra and
algebraic geometry as can be found in, e.g., \cite{Harris92},
\cite{Kunz85} or \cite{Shafarevich94}.

Let $\K$ be any of the fields $\fq$ or $\cfq$. We say that $V$ is a
projective (affine) variety defined over $\K$ if it is the set of
all common zeros in $\Pp^n$ ($\A^n$) of a family of homogeneous
polynomials $F_1\klk F_m \in\K[\xon]$ (of polynomials $F_1\klk F_m
\in \K[X_1 \klk X_n]$). In the remaining part of this section,
unless otherwise stated, all results referring to varieties in
general should be understood as valid for both projective and affine
varieties.

For a $\K$--variety $V$ in the $n$--dimensional (affine or
projective) space, we denote by $I(V)$ its defining ideal and by
$\K[V]$ its coordinate ring. The {\em dimension} $\dim V$ of a
$\K$--variety $V$ is the (Krull) dimension of its coordinate ring
$\K[V]$.  The {\em degree} $\deg V$ of an irreducible $\K$--variety
$V$ is the maximum number of points lying in the intersection of $V$
with a generic linear subspace $L$ of codimension $\dim V$ for which
$|V\cap L|<\infty$ holds. More generally, if $V=V_1\cup\cdots\cup
V_s$ is the decomposition of $V$ into irreducible $\K$--components,
we define the degree of $V$ as $\deg V:=\sum_{i=1}^s\deg V_i$ (cf.
\cite{Heintz83}). We shall say that $V$ has {\em pure dimension} $r$
if every irreducible $\K$--component of $V$ has dimension $r$. A
$\K$--variety $V $ is {\em absolutely irreducible} if it is
irreducible as $\cfq$--variety.

We say that a $\K$--variety $V$ of dimension $r$ in the
$n$--dimensional space is an (ideal--theoretic) {\em complete
intersection} if its ideal $I(V)$ over $\K$ can be generated by
$n-r$ polynomials. If $V\subset\Pp^n$ is a complete intersection
defined over $\K$, of dimension $r$ and degree $\delta$, and $F_1
\klk F_{n-r}$ is a system of generators of $I(V)$, the degrees
$d_1\klk d_{n-r}$ depend only on $V$ and not on the system of
generators. Arranging the $d_i$ in such a way that $d_1\geq d_2 \geq
\cdots \geq d_{n-r}$, we call $\boldsymbol{d}:=(d_1\klk d_{n-r})$
the {\em multidegree} of $V$. In particular, it follows that
$\delta= \prod_{i=1}^{n-r}d_i$ holds.


An important tool for our estimates is the following {\em B\'ezout
inequality} (see \cite{Heintz83},
\cite{Fulton84}, \cite{Vogel84}): if $V$ and $W$ are
$\K$--varieties, then the following inequality holds:
  \begin{equation}\label{equation:Bezout}
    \deg (V\cap W)\le \deg V \cdot \deg W.
  \end{equation}

%

Let $\phi:V\to W$ be a regular linear map of $\K$--varieties. Then
we have (see, e.g., \cite[Lemma 2.1]{CaMa07}):
  \begin{equation}\label{eq:degree linear projection}
    \deg \overline{\phi (V)} \leq \deg V.
  \end{equation}

For a given variety $V$, we denote by $V(\fq)$ the set of
$\fq$--rational points of $V$, namely, $V(\fq):=V\cap \Pp^n(\fq)$ in
the projective case and $V(\fq):=V\cap \A^n(\fq)$ in the affine
case. For a projective variety $V$ of dimension $r$ and degree
$\delta$ we have the upper bound (see \cite[Proposition
12.1]{GhLa02} or \cite[Proposition 3.1]{CaMa07})
 \begin{equation}\label{eq: projective upper bound}
   |V(\fq)|\leq \delta p_r.
 \end{equation}
%
%
\subsection{Multiprojective space}
Let $\N:=\Z_{\ge 0}$ be the set of nonnegative integers. For
$\boldsymbol{n}:=(n_1\klk n_m)\in\N^m$, we define
$|\boldsymbol{n}|:=n_1\plp n_m$ and $\boldsymbol{n}!:=n_1!\cdots
n_m!$. Given $\boldsymbol{\alpha},\boldsymbol{\beta}\in\N^m$, we
write $\boldsymbol{\alpha}\ge\boldsymbol{\beta}$ whenever
$\alpha_i\ge\beta_i$ holds for $1\le i\le m$. For
$\boldsymbol{d}:=(d_1\klk d_m)\in\N^m$, the set
$\N_{\boldsymbol{d}}^{\boldsymbol{n+1}}:=\N_{d_1}^{n_1+1}\times\cdots\times
\N_{d_m}^{n_m+1}$ consists of the elements
$\boldsymbol{a}:=(\boldsymbol{a}_1\klk\boldsymbol{a}_m)\in\N^{n_1+1}\times\cdots
\times\N^{n_m+1}$ with $|\boldsymbol{a}_i|=d_i$ for $1\le i\le m$.

For $\K:=\cfq$ or $\K:=\fq$, we denote by $\Pp^{\boldsymbol{n}}(\K)$
the multiprojective space $\Pp^{\boldsymbol{n}}(\K):=
\Pp^{n_1}(\K)\times\cdots\times \Pp^{n_m}(\K)$ defined  over $\K$.
We shall use the notation
$\Pp^{\boldsymbol{n}}:=\Pp^{\boldsymbol{n}}(\cfq)$. For $1\le i\le
m$, let $\boldsymbol{X}_i:=\{X_{i,0}\klk X_{i,n_i}\}$ be group of
$n_i+1$ variables and let $\boldsymbol{X}:=\{\boldsymbol{X}_1\klk
\boldsymbol{X}_m\}$. A {\em multihomogeneous} polynomial
$F\in\K[\boldsymbol{X}]$ of multidegree $\boldsymbol{d}:=(d_1\klk
d_m)$ is a polynomial which is homogeneous of degree $d_i$ in
$\boldsymbol{X}_i$ for $1\le i\le m$. An ideal
$I\subset\K[\boldsymbol{X}]$ is {\em multihomogeneous} if it is
generated by a family of multihomogeneous polynomials. For any such
ideal, we denote by $V(I)\subset\Pp^{\boldsymbol{n}}$ the variety
defined over $\K$ as its set of common zeros. In particular, an
hypersurface in $\Pp^{\boldsymbol{n}}$ defined over $\K$ is the set
of zeros of a multihomogeneous polynomial of $\K[\boldsymbol{X}]$.
The notions of irreducible variety and dimension of a subvariety of
$\Pp^{\boldsymbol{n}}$ are defined as in the projective space.

Let $V\subset\Pp^{\boldsymbol{n}}$ be an irreducible variety of
dimension $r$ and let $I(V)\subset\cfq[\boldsymbol{X}]$ be its
multihomogeneous ideal. The quotient ring
$\cfq[\boldsymbol{X}]/I(V)$ is multigraded and its part of
multidegree $\boldsymbol{b}\in\N^m$ is denoted by
$(\cfq[\boldsymbol{X}]/I(V))_{\boldsymbol{b}}$. The {\em
Hilbert--Samuel} function of $V$ is the function $H_V:\N^m\to\N$
defined as $H_V(\boldsymbol{b}):=\dim
(\cfq[\boldsymbol{X}]/I(V))_{\boldsymbol{b}}$. It turns out that
there exist $\boldsymbol{\delta_0}\in\N^m$ and a unique polynomial
$P_V\in\Q[z_1\klk z_m]$ of degree $r$ such that
$P_V(\boldsymbol{\delta})=H_V(\boldsymbol{\delta})$ for every
$\boldsymbol{\delta}\in\N^m$ with $\boldsymbol{\delta}\ge
\boldsymbol{\delta_0}$ (see, e.g., \cite[Proposition
1.8]{DaKrSo13}). For $\boldsymbol{b}\in\N^m_r$, we define the {\em
mixed degree of $V$ of index} $\boldsymbol{b}$ as the nonnegative
integer
$$\deg_{\boldsymbol{b}}(V):=\boldsymbol{b}!\,\mathrm{coeff}_{\boldsymbol{b}}(P_V).$$
This notion can be extended to equidimensional varieties and, more
generally, to equidimensional cycles (formal linear combination with
integer coefficients of subvarieties of equal dimension) by
linearity.

The Chow ring of $\Pp^{\boldsymbol{n}}$ is the graded ring
$$A^*(\Pp^{\boldsymbol{n}}):=\Z[\theta_1\klk\theta_m]/(\theta_1^{n_1+1}
\klk \theta_m^{n_m+1}),$$
where each $\theta_i$ denotes the class of the inverse image of a
hyperplane of $\Pp^{n_i}$ under the projection
$\Pp^{\boldsymbol{n}}\to\Pp^{n_i}$. Given a variety
$V\subset\Pp^{\boldsymbol{n}}$ of pure dimension $r$, its class in
the Chow ring is
$$[V]:=\sum_{\boldsymbol{b}}\deg_{\boldsymbol{b}}(V) \theta_1^{n_1-b_1}
\cdots \theta_m^{n_m-b_1}\in A^*(\Pp^{\boldsymbol{n}}),$$
where the sum is over all $\boldsymbol{b}\in\N^m_r$ with
$\boldsymbol{b}\le \boldsymbol{n}$.  This is an homogeneous element
of degree $|\boldsymbol{n}|-r$. In particular, if
$\mathcal{H}\subset\Pp^{\boldsymbol{n}}$ is an hypersurface and
$F\in\cfq[\boldsymbol{X}]$ is a polynomial of minimal degree
defining $\mathcal{H}$, then
$$[\mathcal{H}]:=\sum_{i=1}^m\deg_{\boldsymbol{X}_i}(F)\,\theta_i$$
(see, e.g., \cite[Proposition 1.10]{DaKrSo13}).
%
%
\section{Number of zeros of multihomogeneous
hypersurfaces}\label{section: zeros multih hypersurfaces}
Let $\boldsymbol{n}:=(n_1\klk n_m)\in\N^m$ and let
$\Pp^{\boldsymbol{n}}(\fq)$ be the multiprojective space defined
over $\fq$. For $1\le i\le m$, let $\boldsymbol{X}_i:=\{X_{i,0}\klk
X_{i,n_i}\}$ be a group of $n_i+1$ variables and let
$\boldsymbol{X}:=\{\boldsymbol{X}_1\klk \boldsymbol{X}_m\}$. Let
$F\in\cfq[\boldsymbol{X}]$ be a multihomogeneous polynomial of
multidegree $\boldsymbol{d}:=(d_1\klk d_m)$. In this section we
collect basic facts concerning the number of $\fq$--rational zeros
of $F$.

For $\boldsymbol{\alpha}\in\N^m$, we use the notations
$\boldsymbol{d}^{\boldsymbol{\alpha}}:=d_1^{\,\alpha_1}\cdots
d_m^{\,\alpha_m}$ and
$p_{\boldsymbol{n}-\boldsymbol{\alpha}}:=p_{n_1-\alpha_1}\cdots
p_{n_m-\alpha_m}$ for $\boldsymbol{n}\ge \alpha$. We have the
following result.
\begin{proposition}\label{prop: upper bound multihomogeneous}
Let $F\in\cfq[\boldsymbol{X}]$ be a multihomogeneous polynomial of
multidegree $\boldsymbol{d}$ with $\max_{1\le i\le m}d_i<q$ and let
$N$ be the number of zeros of $F$ in $\Pp^{\boldsymbol{n}}(\fq)$.
Then
$$N\le \eta_m(\boldsymbol{d},\boldsymbol{n})
:=\sum_{\boldsymbol{\varepsilon}\in\{0,1\}^m\setminus\{\boldsymbol{0}\}}
(-1)^{|\boldsymbol{\varepsilon}|+1}\boldsymbol{d}^{\,
\boldsymbol{\varepsilon}}p_{\boldsymbol{n}-\boldsymbol{\varepsilon}}.$$
\end{proposition}
\begin{proof}
We argue by induction on $m$. The case $m=1$ of the statement is
(\ref{eq: projective upper bound}).

Suppose that the statement holds for $m-1$ and let
$F\in\cfq[\boldsymbol{X}]$ be an $m$--homogeneous polynomial of
multidegree $\boldsymbol{d}:=(d_1\klk d_m)$. Let $N$ be the number
of zeros of $F$ in $\Pp^{\boldsymbol{n}}(\fq)$, and let $Z_m$ be the
subset of $\Pp^{n_m}(\fq)$ which consists of the elements
$\boldsymbol{x}_m$ such that the substitution
$F(\boldsymbol{X}_1\klk \boldsymbol{X}_{m-1},\boldsymbol{x}_m)$ of
$\boldsymbol{x}_m$ for $\boldsymbol{X}_m$ in $F$ yields the zero
polynomial of $\cfq[\boldsymbol{X}_1\klk\boldsymbol{X}_{m-1}]$.
According to (\ref{eq: projective upper bound}) we have $|Z_m|\le
d_mp_{n_m-1}$. Fix $\boldsymbol{x}_m\in\Pp^{n_m}(\fq)\setminus Z_m$
and denote by $N_{m-1}$ the number of zeros of
$F(\boldsymbol{X}_1\klk \boldsymbol{X}_{m-1},\boldsymbol{x}_m)$ in
$\Pp^{n_1}(\fq)\times\cdots\times \Pp^{n_{m-1}}(\fq)$. Then the
inductive hypothesis implies
$$N_{m-1}\le \eta_{m-1}(\boldsymbol{d}^*,\boldsymbol{n}^*),$$
where $\boldsymbol{d}^*:=(d_1\klk d_{m-1})$ and
$\boldsymbol{n}^*:=(n_1\klk n_{m-1})$. As a consequence, we obtain
\begin{eqnarray*}
N &\le& |Z_m|\cdot p_{n_1}\cdots p_{n_{m-1}}+(p_{n_m}-|Z_m|)
\cdot\eta_{m-1}(\boldsymbol{d}^*,\boldsymbol{n}^*)\\
&=& |Z_m|\left(p_{n_1}\cdots
p_{n_{m-1}}-\eta_{m-1}(\boldsymbol{d}^*,\boldsymbol{n}^*)\right)+
\eta_{m-1}(\boldsymbol{d}^*,\boldsymbol{n}^*)\cdot p_{n_m}\\&\le&
\eta_m(\boldsymbol{d},\boldsymbol{n}).
\end{eqnarray*}
This completes the proof of the proposition.
\end{proof}

Since the proof of Proposition \ref{prop: upper bound
multihomogeneous} is concerned with hypersurfaces, in order to bound
from above the number of $\fq$--rational points of a projective
hypersurface $\mathcal{H}\subset\Pp^n$ of degree $\delta$ one may
use the Serre bound (see \cite{Serre91}):
\begin{equation}\label{eq: upper bound Serre}
|\mathcal{H}(\fq)|\le \delta q^{n-1}+p_{n-2}.
\end{equation}
Although (\ref{eq: upper bound Serre}) is stated for hypersurfaces
defined over $\fq$ in \cite{Serre91}, it is easy to see that it also
holds for hypersurfaces defined over $\cfq$. Using (\ref{eq: upper
bound Serre}) the upper bound of Proposition \ref{prop: upper bound
multihomogeneous} can be slightly improved. In particular, it may be
worthwhile to remark that, if $\boldsymbol{d},\boldsymbol{n}\in\N^m$
are of the form $\boldsymbol{d}=(d\klk d)$ and
$\boldsymbol{n}:=(n\klk n)$, by using (\ref{eq: upper bound Serre})
we obtain
\begin{equation}\label{eq: upper bound Serre multih}
N\le p_n^m-(q^n-(d-1)q^{n-1})^m.
\end{equation}
%

A similar argument as in the proof of Proposition \ref{prop: upper
bound multihomogeneous} yields the following upper bound for the
number $N_a$ of zeros of $F$ in
$\fq^{\boldsymbol{n}+\boldsymbol{1}}:=\fq^{n_1+1}\times
\cdots\times\fq^{n_m+1}$:
\begin{equation}\label{eq: upper bound number points affine multihomogeneous}
N_a\le \eta_m^a(\boldsymbol{d},\boldsymbol{n}):=
\sum_{\boldsymbol{\varepsilon}\in\{0,1\}^m\setminus\{\boldsymbol{0}\}}
(-1)^{|\boldsymbol{\varepsilon}|+1}\boldsymbol{d}^{\,\boldsymbol{\varepsilon}}
\boldsymbol{q}^{\boldsymbol{n}+\boldsymbol{1}-\boldsymbol{\varepsilon}},
\end{equation}
where $\boldsymbol{q},\boldsymbol{1}\in\N^m$ are defined by
$\boldsymbol{q}:=(q\klk q)$ and $\boldsymbol{1}:=(1\klk 1)$.

%
%
%
%
%
%
We end this section with the following consequence of Proposition
\ref{prop: upper bound multihomogeneous}.
\begin{corollary}\label{coro: existencia no-cero pol multih}
Let $F\in\cfq[\boldsymbol{X}]$ be a multihomogeneous polynomial of
multidegree $\boldsymbol{d}$ and let $d:=\max_{1\le i\le m}d_i$. If
$q>d$, then there exists $\boldsymbol{x}\in
\Pp^{\boldsymbol{n}}(\fq)$ such that $F(\boldsymbol{x})\not=0$
holds.
\end{corollary}
\begin{proof}
It suffices to show that there exists
$\boldsymbol{x}\in\fq^{\boldsymbol{n}+\boldsymbol{1}}$ for which
$F(\boldsymbol{x})\not=0$ holds. For this purpose, according to
(\ref{eq: upper bound number points affine multihomogeneous}) we
have that the number of elements of
$\fq^{\boldsymbol{n}+\boldsymbol{1}}$ not annihilating $F$ is
bounded from above by the following quantity:
$$\boldsymbol{q}^{\boldsymbol{n}+\boldsymbol{1}}-\eta_m^a(\boldsymbol{d},\boldsymbol{n})
=\sum_{\boldsymbol{\varepsilon}\in\{0,1\}^m}
(-1)^{|\boldsymbol{\varepsilon}|}\boldsymbol{d}^{\,\boldsymbol{\varepsilon}}
\boldsymbol{q}^{\boldsymbol{n}+\boldsymbol{1}-\boldsymbol{\varepsilon}}=
\boldsymbol{q}^{\boldsymbol{n}}\prod_{j=1}^m(q-d_i).$$
Our hypothesis implies that the right--hand side of the previous
identities is strictly positive, which immediately yields the
corollary.
\end{proof}

%
\section{Polar varieties}\label{section: polar varieties}
Let $V \subset \Pp^n$ be a variety of pure dimension $r$ and degree
$\delta$. Let $\Sigma \subset V$ denote the singular locus of $V$
and let $V_{\rm sm}:=V\setminus\Sigma$. For each integer $s$ with
$0\leq s \leq r-2$  and for $x \in V_{\rm sm}$, a linear variety $L
\subset \Pp^n$ of dimension $n - s -2 $ meets $T_{x}V \subset \Pp^n$
in dimension at least $r-s-2 $. The set of points $x \in V_{\rm sm}$
such that the dimension of the intersection is greater than or equal
to $r-s-1$ is called the $s$th {\em polar variety} of $V$ with
respect to $L$ and is denoted by $\M(L)$. In symbols,
  $$
    \M(L): = \left\{ x \in V_{\rm sm}: \dim (T_{x}V \cap L) \geq r-s-1 \right\}.
  $$

This is a classical notion of  projective geometry. The polar
variety $\M(L)$ is empty or of pure dimension at least $s$. In fact,
following \cite{Kleiman76} we have that, for a generic $L$, the
polar variety $\M(L)$ has dimension $s$. We include a proof of this
result for the sake of completeness (see also \cite[Transversality
Lemma 1.3]{Piene78}).
\begin{proposition}\label{prop: dimension polar variety}
For a generic linear variety $L\subset\Pp^n$ of dimension $n-s-2$,
the polar variety $\M(L)$ has dimension $s$.
\end{proposition}
\begin{proof}
Let $\mathbb{G}(r,n)$ denote the Grassmannian of $r$--planes in
$\Pp^{n}$. We consider the Gauss map $\mathcal{G}:V_{\rm
sm}\to\mathbb{G}(r,n)$, which maps a point $x$ into the tangent
space $T_x V$. Let $S\subset \mathbb{G}(r,n)$ be the following
Schubert variety:
$$S= \left\{\Lambda\in \mathbb{G}(r,n):\,\mathrm{dim}(\Lambda\cap L)\geq r-s-1\right\}.$$
First we observe that $S$ has dimension
$\mathrm{dim}\,\mathbb{G}(r,n)-(r-s)$ (see, e.g., \cite[Example
11.42]{Harris92}). Furthermore, it is clear that
$\M(L)=\mathcal{G}^{-1}(S\cap
\mathcal{G}\big(V_{\mathrm{sm}})\big)$. Let $i:S\hookrightarrow
\mathbb{G}(r,n)$ denote the standard inclusion mapping. We claim
that the polar variety $\M(L)$ coincides with the fiber product
$V_{\rm sm}\times_{\mathbb{G}(r,n)}S$. Indeed,
$$\begin{array}{rcl}
V_{\rm sm}\times_{\mathbb{G}(r,n)}S
  &= & \{(x, \Lambda)\in V_{\rm sm}\times S\,:\, T_x V=\Lambda\} \\
  &= &  \{x\in V_{\rm sm}\,:\,\mathrm{dim}(T_x V\cap L)\geq r-s-1\}
 \ =\ \M(L).
  \end{array}$$

The general linear group acts transitively on $\mathbb{G}(r,n)$, and
with respect to this action $S$ is in general position because $L$
is so by hypothesis. Therefore, \cite[Theorem 2]{Kleiman74} shows
that $\M(L)$ is of pure dimension
$$\mathrm{dim}\, \M(L)=\mathrm{dim}\,V_{\rm sm} +  \mathrm{dim}\, S -\mathrm{dim}
\,\mathbb{G}(r,n) =s.$$
This finishes the proof of the proposition.
\end{proof}

Set $X:=(X_0\klk X_n)$. For $\mu:=(\mu_0:\dots:\mu_n)\in\Pp^n$, we
shall use the notation $\mu\cdot X:=\mu_0X_0\plp\mu_nX_n$. Let
$\lambda_0, \klk \lambda_{s+1}$ be linearly independent elements of
$\Pp^n$ and let $L\subset \Pp^n$ be the linear space of dimension
$n-s-2$ defined by
\begin{equation}\label{eq: definition L}
L:=\{x\in\Pp^n:\lambda_0 \cdot x = \cdots = \lambda_{s+1}\cdot x =
0\}.
\end{equation}
For $x\in V_{\mathrm{sm}}$, let $\varphi_x:T_xV\to\Pp^{s+1}$ be the
linear mapping defined by $Y_i:=\lambda_i\cdot X$ $(0\le i\le s+1)$.
Observe that $\varphi_x$ may be seen as the differential mapping of
the morphism $C_V\to\A^{s+2}$ defined by $Y_0\klk Y_{s+1}$, where
$C_V\subset\A^{n+1}$ is the affine cone of $V$.

The set $\E_x$ of exceptional points of $\varphi_x$ is a linear
subspace of $\Pp^n$ of dimension at least  $r-s-2$ which equals $T_x
V \cap L$. Therefore, we may characterize the polar variety $\M (L)$
in terms of the dimension of $\E_x$ for $x\in V_{\mathrm{sm}}$, as
in the following remark.

\begin{remark}\label{lemma: polar variety through kernels}
The polar variety $\M (L)$ coincides with the set of points $ x \in
V_{\rm sm}$ such that the dimension of $\E_{x}$ is at least $r-s-1$.
\end{remark}

A critical point for our approach is that the polar variety $\M(L)$
can be defined in terms of the vanishing of certain minors involving
the partial derivatives of the polynomials defining $V$ and $L$.
This has the advantage of providing an explicit system of equations
defining the polar variety $\M(L)$. In the series of papers
\cite{BaGiHeMb97}, \cite{BaGiHeMb01}, \cite{BaGiHePa05},
\cite{BaGiHeSaSc10}, \cite{BaGiHeLePa12} polar varieties are locally
described by regular sequences consisting of the polynomials
defining $V$ and certain well--determined maximal minors of their
Jacobian in the context of efficient real elimination.

Assume that $V$ is an ideal--theoretic complete intersection in
$\Pp^n$ defined by homogeneous polynomials $F_1, \ldots, F_{n-r} \in
\fq[\xo{n}]$ of degrees $d_1\ge\dots\ge d_{n-r}\ge 2$ respectively
and let $D:=\sum_{i=1}^{n-r}(d_i-1)$. For any $ x \in V_{\rm sm}$,
the gradients $\nabla F_1(x)\klk\nabla F_{n-r}(x)$ are linearly
independent and the tangent space $T_x V$ is the $r$--dimensional
linear variety
$$T_x V = \big \{ v \in \Pp^n: \nabla F_1(x) \cdot v = \cdots =\nabla
F_{n-r}(x)\cdot v =0 \big\}.$$
Let $0\leq s \leq r-2$ and consider the $(n-s-2 )$-dimensional
linear variety $L$ of (\ref{eq: definition L}). Write
$\lambda_i:=(\lambda_{i,0}\klk \lambda_{i,n})$ for $0\le i\le s+1$,
$\boldsymbol{\lambda} :=(\lambda_0, \ldots, \lambda_{s+1})$ and
consider the matrix
\begin{equation}\label{eq: matrix for polar variety}
    \M(X,\boldsymbol{\lambda}):=
     \begin{pmatrix}
            \frac{\partial F_1}{\partial X_0} &
            \dots & \frac{\partial F_1}{\partial X_n}
            \cr \vdots & & \vdots \cr
            \frac{\partial F_{n-r}}{\partial X_0} &
            \dots & \frac{\partial F_{n-r}}{\partial X_n} \cr
            \lambda_{0,0} & \dots & \lambda_{0,n} \cr
            \vdots & & \vdots \cr
            \lambda_{s+1,0} & \dots & \lambda_{s+1,n}
     \end{pmatrix}.
\end{equation}
The dimension of $T_{x}V \cap L$ is equal to $ r-s-2$ if and only if
$\M(x,\boldsymbol{\lambda})$ has maximal rank. Equivalently,
$\M(x,\boldsymbol{\lambda})$ has not maximal rank if and only if the
dimension of $T_{x}V \cap L$ is at least $r-s-1$. As a consequence,
if we denote by $\Delta_{1}(x,\boldsymbol{\lambda}), \ldots,
\Delta_{N}(x,\boldsymbol{\lambda})$ the maximal minors of $\M(x,
\boldsymbol{\lambda})$, then the polar variety ${\sf M}(L)$ is given
by
$$
{\sf M}(L) = \{x \in V_{\rm sm}:
\Delta_1(x,\boldsymbol{\lambda})=\dots=\Delta_N(x,\boldsymbol{\lambda})=0\}.
$$
\begin{proposition}\label{prop: polar variety as a subset of a variety of dimension s}
Suppose that  $\M(L)$ has dimension $s$ and denote by $\Sigma$ the
singular locus of $V$. Then there exists a subvariety $Z(L)\subset
V$ of pure dimension $s$ and degree at most $D^{r-s}
  \delta$ such that $\M(L)\cup\Sigma\subset Z(L)$ holds.
\end{proposition}
\begin{proof}
Since $\M(L)\cup\Sigma=\{x\in V:\Delta_1(x,\boldsymbol{\lambda})
=\dots=\Delta_N(x,\boldsymbol{\lambda})=0\}$ has dimension at most
$s\le r-2$, there exists $x\in V\setminus(\M(L)\cup\Sigma)$. For
such an $x$, there exists at least a maximal minor $\Delta_j$ of the
matrix $\M(X,\boldsymbol{\lambda})$ of (\ref{eq: matrix for polar
variety}) with $\Delta_j(x,\boldsymbol{\lambda})\not=0$, and thus an
$\cfq$--linear combination $G_1:=\sum_{j=1}^{N}\gamma_j^{(1)}
\Delta_j$ with $G_1(x,\boldsymbol{\lambda}) \not=0$. This implies
that $G_1\in \cfq[X_0,\ldots, X_n]$ is a nonzero polynomial of
degree $D$ vanishing on $\M(L)\cup\Sigma$ and not vanishing
identically on $V$. The absolute irreducibility of $V$ implies that
$V^{(1)}:=V\cap \{G^{(1)}=0\}$ is a projective variety of pure
dimension $r-1$ for which $\M(L)\cup\Sigma\subset V^{(1)}$ holds. By
the B\'ezout inequality (\ref{equation:Bezout}) we deduce that $\deg
V^{(1)}\leq D \delta$.

Let $V^{(1)}=\bigcup_{i=1}^t\mathcal{C}_i$ be the decomposition of
$V^{(1)}$ into absolutely irreducible components. Since $\dim(\M
(L)\cup\Sigma)\leq s<r-1$ holds, we may choose regular points $x_i
\in \mathcal{C}_i\setminus \M (L)$ for $1\le i\le t$. Arguing as
above, we conclude that there exist $\gamma_1^{(2)}, \ldots,
\gamma_{N}^{(2)}\in\cfq$ such that the nonzero polynomial
$G^{(2)}:=\sum_{j=1}^{N}\gamma_j^{(2)}\Delta_j(X, \lambda)$ does not
vanish on $x_i$ for $1\le i\le t$. Therefore $V^{(2)}:=V^{(1)} \cap
\{G^{(2)}=0\}$ is a projective variety of pure dimension $r-2$ and
degree at most $D^2\delta$ with $\M(L)\cup\Sigma \subset V^{(2)}$.

Applying successively this argument we finally obtain a projective
variety $V^{(r-s)}$ of pure dimension $s$ and degree at most
$D^{r-s}\delta$ with $\M(L)\cup \Sigma\subset V^{(r-s)}$. The proof
of the proposition finishes setting $Z(L):= V^{(r-s)}$.
 \end{proof}

In the sequel we shall encounter several times a similar situation
as in Proposition \ref{prop: polar variety as a subset of a variety
of dimension s}, namely  a projective or multiprojective subvariety
$W_1$ of a pure dimensional variety $W$, which is defined as the
zero locus in $W$ of homogeneous or multihomogeneous polynomials
$H_1\klk H_M$. If $m$ denotes the codimension of $W_1$ in $W$,
arguing as in the proof of Proposition \ref{prop: polar variety as a
subset of a variety of dimension s} we shall conclude that there
exist $m$ generic linear combinations $H^1\klk H^m$ of $H_1\klk H_M$
such that the zero locus $W_2$ of $H^1\klk H^m$ in $W$ has pure
codimension $m$ in $W$ and contains $W_1$.

As established in Proposition \ref{prop: dimension polar variety},
for  a generic choice of $L$ the dimension of ${\sf M}(L)$ is equal
to $s$. Our next goal is to obtain conditions on $\lambda_0, \dots,
\lambda_{s+1} \in \Pp^n$ which imply that the polar variety $\M(L)$
has dimension $s$. For $0\le i\le s+1$ we shall denote by
$\Lambda_i:=(\Lambda_{i,0}\klk \Lambda_{i,n})$ a group of $n+1$
variables and set
$\boldsymbol{\Lambda}:=(\Lambda_0\klk\Lambda_{s+1})$. We consider
the so--called generic polar variety, namely
$$
W:=(V_{\rm sm} \times
\mathcal{U})\cap\{\Delta_1(X,\boldsymbol{\Lambda}) = \dots =
\Delta_N(X,\boldsymbol{\Lambda})=0\},
$$
where $\mathcal{U}\subset(\Pp^n)^{s+2}$ is the Zariski open subset
consisting of all the $(s + 2)\times (n+1)$--matrices of maximal
rank and $\Delta_1\klk\Delta_N$ are the maximal minors of the
generic version $\M(X,\boldsymbol{\Lambda})$ of the matrix
$\M(X,\boldsymbol{\lambda})$ of (\ref{eq: matrix for polar
variety}).
\begin{proposition}\label{prop: generic polar var has dimension s+t}
Let $t:=n(s+2)$. Then $W$ is an irreducible variety of $\Pp^n\times
\mathcal{U}$ of dimension $s + t$.
\end{proposition}
\begin{proof}
Let $\pi_1:W\to V_{\rm sm}$ be the linear projection
$\pi_1(x,\boldsymbol{\lambda}):= x$. Fix $x\in V_{\rm sm}$ and
consider the fiber $\pi_1^{-1}(x)$. We have that
$\pi_1^{-1}(x)=\{x\}\times \mathcal{L}$, where $\mathcal{L} \subset
\mathcal{U}$ denotes the set of matrices
$\boldsymbol{\lambda}:=(\lambda_0,\dots,\lambda_{s+1})$ for which
the matrix ${\sf M}(x,\boldsymbol{\lambda})$ is not of full rank.
This is the same as saying that
$$
\langle\lambda_0,\dots,\lambda_{s+1}\rangle \cap \big\langle\nabla
F_1(x)\klk\nabla F_{n-r}(x)\big\rangle \neq \emptyset,
$$
where $\langle v_0\klk v_m\rangle \subset \A^{n+1}$ is the linear
variety spanned by $v_0\klk v_m$ in $\A^{n+1}$. Equivalently, the
vectors $\lambda_0\klk\lambda_{s+1}$ are not linearly independent in
the quotient $\cfq$--vector space
$$
\mathbb{V}:=\A^{n+1}/\big(\nabla F_1(x)\klk\nabla F_{n-r}(x)\big).
$$
This shows that $\mathcal{L}$ is, modulo $\big(\nabla
F_1(x)\klk\nabla F_{n-r}(x)\big)$, isomorphic to the Zariski open
set $L_{s+1}(\A^{s+2},\mathbb{V})\cap \mathcal{U}$, where
$$
L_{s+1}(\A^{s+2},\mathbb{V}):=\{f\in{\rm
Hom}_{\cfq}(\A^{s+2},\mathbb{V}):{\rm rank}(f)\le s+1\}.
$$
According to \cite[Proposition 1.1]{BrVe88},
$L_{s+1}(\A^{s+2},\mathbb{V})$ is an irreducible variety of
dimension $(s+1)(r+2)$. Taking into account that we are considering
subspaces of $\A^{n+1}$ of dimension $s+2$ modulo a subspace
$\langle\nabla F_1(x)\klk\nabla F_{n-r}(x)\rangle$ of dimension
$n-r$, we see that $\mathcal{L}$ is an irreducible variety of
$\Pp^n\times \A^{(n+1)(s+2)}$ of dimension
$(s+1)(r+2)+(n-r)(s+2)=(n+1)(s+2)+s-r$. We may also rephrase this
conclusion saying that $\pi_1^{-1}(x) =\{x\}\times \mathcal{L}$ is
an irreducible variety of $(\Pp^n)^{s+3}$ of dimension $t+s-r$

Our previous arguments shows that $\pi_1: W\to V_{\rm sm}$ is
surjective. Then the proof of \cite[\S I.6.3, Theorem
8]{Shafarevich94} shows that $W$ is an irreducible variety of
$V_{\rm sm}\times\mathcal{U}$ of dimension $s+t$.
\end{proof}

In the sequel, we shall associate each point
$\boldsymbol{\lambda}:=(\lambda_0\klk
\lambda_{s+1})\in(\Pp^n)^{s+2}$ with the linear space
$L:=\{x\in\Pp^n:\lambda_0\cdot x=\dots=\lambda_{s+1}\cdot x=0\}$.
\begin{theorem} \label{theorem: dimension polar variety intrinseca}
There exists an hypersurface $\mathcal{H}_1 \subset (\Pp^n)^{s+2}$,
defined by a multihomogeneous polynomial of degree at most
$(n-s)(r-s)D^{r-s-1}\delta+1$ in each group of variables
$\Lambda_i$, such that for any
$\boldsymbol{\lambda}\in(\Pp^n)^{s+2}\setminus \mathcal{H}_1$ the
polar variety $\M(L)$ has dimension at most $s$.
\end{theorem}
\begin{proof}
According to Proposition \ref{prop: dimension polar variety}, for a
generic matrix $\boldsymbol{\lambda}\in\mathcal{U}$ the polar
variety $\M(L)$ is of pure dimension $s\ge 0$. Then the projection
mapping $\pi_2:W \to (\Pp^n)^{s+2}$ defined by
$\pi_2(x,\boldsymbol{\lambda}):= \boldsymbol{\lambda}$ is dominant.
This implies that the field extension
$\cfq(\boldsymbol{\Lambda})\hookrightarrow \cfq(W)$ has
transcendence degree $s+1$, and therefore, there exist indices
$i_0\klk i_s$ such that the coordinate functions of $\cfq(W)$
defined by $X_{i_0}\klk X_{i_s}$ form a transcendence basis of this
field extension.

Fix $i\in\Gamma:=\{0\klk n\}\setminus\{i_0\klk i_s\}$ and consider
the linear mapping $\pi^i:W\to\Pp^{s+1}\times(\Pp^n)^{s+2}$ defined
by $X_{i_0}\klk X_{i_s},X_i$ and $\boldsymbol{\Lambda}$. Then the
Zariski closure $W_i\subset \Pp^{s+1}\times(\Pp^n)^{s+2}$ of
$\pi^i(W)$ is an hypersurface. Since $F_1\klk F_{n-r}$ define a
subvariety of $(\Pp^n)^{s+3}$ of pure dimension $r+n(s+2)$, from
Proposition \ref{prop: generic polar var has dimension s+t} we
conclude that there exist $r-s$ generic linear combinations, say
$\Delta^1\klk\Delta^{r-s}$, of $\Delta_1\klk\Delta_N$, such that
$F_1\klk F_{n-r},\Delta^1\klk\Delta^{r-s}$ define a subvariety $W'$
of $(\Pp^n)^{s+3}$ of pure dimension $s+n(s+2)$ containing $W$. In
particular, the Zariski closure $W_i'$ of $\pi^i(W')$ is an
hypersurface of $\Pp^{s+1}\times(\Pp^n)^{s+2}$ which contains $W_i$.

Next we estimate the multidegree of $W_i'$, and hence of $W_i$. For
this purpose, we consider the class $[W']$ of $W'$ in the Chow ring
$\mathcal{A}^*((\Pp^n)^{s+3})$ of $(\Pp^n)^{s+3}$. Denote by
$\theta_{j-2}$ the class of the inverse image of a hyperplane of
$\Pp^n$ under the $j$th canonical projection $(\Pp^n)^{s+3}\to\Pp^n$
for $1\le j\le s+3$. According to the multihomogeneous B\'ezout
theorem (see, e.g., \cite[Theorem 1.11]{DaKrSo13}), we have
\begin{eqnarray*}
[W']&=&\prod_{i=1}^{n-r}(d_i\theta_{-1})
\prod_{k=1}^{r-s}(D\theta_{-1}+
\theta_0\plp\theta_{s+1})\\
&=&\delta
D^{r-s-1}\left(D(\theta_{-1})^{n-s}+(r-s)(\theta_{-1})^{n-s-1}(
\theta_0\plp\theta_{s+1})\right)\\&&+\mathcal{O}\big((\theta_{-1})^{n-s-2}\big),
\end{eqnarray*}
where $\mathcal{O}\big((\theta_{-1})^{n-s-2}\big)$ represents a sum
of terms of degree at most $n-s-2$ in $\theta_{-1}$.
On the other hand, by definition
$[W_i']=\deg_X\!m_i'\,\theta_{-1}+\deg_{\Lambda_0}\!m_i'\,\theta_0
\plp\deg_{\Lambda_{s+1}}\!m_i'\,\theta_{s+1}$, where
$m_i'\in\fq[X_{i_0}\klk X_{i_s}, X_i,\boldsymbol{\Lambda}]$ is a
polynomial of minimal degree defining $W_i'$. Let
$\jmath:\mathcal{A}^*\big(\Pp^{s+1}\times (\Pp^n)^{s+2}\big)
\hookrightarrow \mathcal{A}^*\big((\Pp^n)^{s+3}\big)$ be the
injective $\Z$--map $P\mapsto(\theta_{-1})^{n-s-1}P$ induced by
$\pi^i$. Then \cite[Proposition 1.16]{DaKrSo13} shows that
$\jmath([W_i'])\le [W']$, where the inequality is understood in a
coefficient--wise sense. This implies
$\deg_{\Lambda_j}\!m_i'\le(r-s)D^{r-s-1}\delta$ for $0\le j\le s+1$.

Let $m_i\in\fq[X_{i_0}\klk X_{i_s}, X_i, \boldsymbol{\Lambda}]$ be
the polynomial defining $W_i$. Observe that $D_i:=\deg_{X_i}m_i>0$
holds. Let $A_i\in\fq[X_{i_0}\klk X_{i_s},\boldsymbol{\Lambda}]$ be
the (nonzero) polynomial arising as the coefficient of $X_i^{D_i}$
in $m_i$, considered as an element of the polynomial ring
$\fq[X_{i_0}\klk X_{i_s},\boldsymbol{\Lambda}][X_i]$. Further, let
$A_i^*\in\fq[\boldsymbol{\Lambda}]$ be a nonzero coefficient of
$A_i$, considering $A_i$ as an element of the polynomial ring
$\fq[\boldsymbol{\Lambda}][X_{i_0}\klk X_{i_s}]$. Finally, let
$A_0\in\fq[\boldsymbol{\Lambda}]$ denote an arbitrary maximal minor
of the generic matrix $(\Lambda_{i,j})_{0\le i\le s+1,0\le j\le n}$
and set
$A:=A_0\cdot\prod_{i\in\Gamma}A_i^*\in\fq[\boldsymbol{\Lambda}]$. We
claim that the hypersurface $\mathcal{H}_1\subset(\Pp^n)^{s+2}$
defined by the zero locus of $A$ satisfies the requirements of the
theorem.

In order to show this claim, let
$\boldsymbol{\lambda}:=(\lambda_0\klk\lambda_{s+1})\in(\Pp^n)^{s+2}
\setminus\mathcal{H}_1$ and denote $m_i^{(\boldsymbol{\lambda})}:=
m_i(X_{i_0}\klk X_{i_s},X_i,\boldsymbol{\lambda})$. Since
$A_0(\boldsymbol{\lambda})\not=0$ holds, we have that
$\boldsymbol{\lambda}\in\mathcal{U}$. Then
$m_i^{(\boldsymbol{\lambda})}$ is a nonzero polynomial of
$\cfq[X_{i_0}\klk X_{i_s},X_i]$ with
$\deg_{X_i}m_i^{(\boldsymbol{\lambda})}>0$ vanishing on $\M(L)$ for
any $i\in\Gamma$, where $L$ is the linear variety associated with
$\boldsymbol{\lambda}$. This implies that the coordinate function of
$\M(L)$ defined by $X_i$ satisfies a nontrivial algebraic equation
over $\cfq(X_{i_0}\klk X_{i_s})$ for any $i\in\Gamma$. As a
consequence, it follows that $\M(L)$ has dimension at most $s$.

Since $A_i^*$ is a multihomogeneous polynomial of
$\cfq[\boldsymbol{\Lambda}]$ with $\deg_{\Lambda_i}A_i^*\le
(r-s)D^{r-s-1}\delta$ and $|\Gamma|=n-s$ holds, we obtain the upper
bound $\deg_{\Lambda_i}A\le (n-s)(r-s)D^{r-s-1}\delta+1$. This
finishes the proof of the theorem.
\end{proof}
%
%
\section{On the existence of nonsingular linear sections}
In this section we shall establish a Bertini--type theorem, namely
we shall show the existence of nonsingular linear sections of a
given variety. Combining the main result of this section and Theorem
\ref{theorem: dimension polar variety intrinseca} we shall be able
to obtain an effective Bertini smoothness theorem suitable for our
purposes.

A version of the Bertini theorem asserts that a generic hyperplane
section of a nonsingular variety $V$ is nonsingular. A more precise
variant of this result asserts that, if $V\subset\Pp^n$ is a
projective variety with a singular locus of dimension at most $s$,
then a section of $V$ defined by a generic linear space of $\Pp^n$
of codimension at least $s+1$ is nonsingular (see, e.g.,
\cite[Proposition 1.3]{GhLa02a}). In this section, we shall consider
the existence of nonsingular sections of codimension $s+2$.
Identifying each section of this type with a point in the
multiprojective space $(\Pp^n)^{s+2}$, we shall show the existence
of an hypersurface of $(\Pp^n)^{s+2}$ containing all the linear
subvarieties of codimension $s+2$ of $(\Pp^n)^{s+2}$ which yield
singular sections of $V$. We shall further provide an estimate of
the multidegree of this hypersurface.

We remark that an effective version of a weak form of a Bertini
theorem is obtained in \cite{Ballico03}. Nevertheless, the bound
given in \cite{Ballico03} is exponentially higher than ours and
therefore is not suitable for our purposes.

Let $V \subset \Pp^n$ be an ideal--theoretic complete--intersection
defined over $\fq$, of dimension $r$ and degree $\delta$. Let
$F_1\klk F_{n-r} \in \fq[\xon]$ be homogeneous polynomials of
degrees $d_1\ge\ldots\ge d_{n-r}\ge 2$ respectively, which generate
the ideal $I(V)$ of $V$. Let $\Sigma \subset V$ be the singular
locus of $V$ and suppose that it has dimension at most $s\le r-2$.
As asserted above, our goal is to obtain a condition on $\lambda_0,
\dots, \lambda_{s+1} \in \Pp^n$ which implies that $V\cap L$ is a
nonsingular variety of dimension $r-s-2$, where $L:=\{\lambda_i\cdot
X=0:0\le i\le s+1\}\subset\Pp^n$.

We recall the notations $D:=\sum_{i=1}^{n-r}(d_i-1)$ and
$t:=n(s+2)$. We start with two technical results.
\begin{lemma}\label{lemma: dimension de la fibra}
There exists an hypersurface $\mathcal{H}_2'\subset(\Pp^n)^{s+2}$,
defined by a multihomogeneous polynomial of
$\cfq[\boldsymbol{\Lambda}]$ of multidegree at most $\delta$ in each
group of variables $\Lambda_i$, with the following property: let
$\boldsymbol{\lambda}\in(\Pp^n)^{s+2}\setminus \mathcal{H}_2'$, let
$(Y_0\klk Y_{s+1}):=\boldsymbol{\lambda}\cdot X$, and let
$\pi:V\to\Pp^{s+1}$ be the linear mapping defined by $Y_0\klk
Y_{s+1}$. Then the Zariski closure $V_y$ of any fiber $\pi^{-1}(y)$
is of pure dimension $r-s-1$ and the set of exceptional points of
$\pi$ is of pure dimension $r-s-2$.
\end{lemma}
\begin{proof}
Let $U_0, \ldots, U_r$ be $r+1$ groups of $n+1$ indeterminates over
$\cfq[\xon]$, where  $U_{i}:=(U_{i,0}, \ldots, U_{i,n})$, and let
$\boldsymbol{U}:=(U_0\klk U_r)$. Denote by
$\mathcal{F}_V\in\fq[\boldsymbol{U}]$ the Chow form of $V$ (see,
e.g., \cite{HoPe68b}, \cite{Samuel67}). This is an irreducible
polynomial of $\cfq[\boldsymbol{U}]$ which characterizes the set of
overdetermined linear systems over $V$. Furthermore, $\mathcal{F}_V$
is homogeneous in each group of variables $U_i$ and satisfies the
identities $\deg_{U_{i,0}}\mathcal{F}_V=\deg_{U_i} \mathcal{F}_V=
\delta$  for $0\le i\le r$.

Consider $\mathcal{F}_V$ as a polynomial of $\fq[U_0\klk
U_{s+1}][U_{s+2}\klk U_r]$ and fix\linebreak $u_{s+2}\klk
u_r\in\Pp^n$ such that $B:=\mathcal{F}_V(U_0\klk U_{s+1},u_{s+2}\klk
u_r)$ does not vanish. We claim that any
$\boldsymbol{\lambda}:=(\lambda_0\klk\lambda_{s+1})\in(\Pp^n)^{s+2}$
with $B(\boldsymbol{\lambda})\not=0$ satisfies the requirements of
the lemma.

Indeed, by the definition of $\boldsymbol{\lambda}$ and
$\boldsymbol{u}:=(u_{s+2}\klk u_r)$ we have that the mapping
$\pi_r:V\to\Pp^r$ defined by the linear forms $\lambda_0\cdot X\klk
\lambda_{s+1}\cdot X,u_{s+2}\cdot X\klk u_r\cdot X$ is a finite
morphism. Let $\pi:V\to\Pp^{s+1}$ be the mapping defined by
$\lambda_0\cdot X\klk\lambda_{s+1}\cdot X$. Then the Zariski closure
$V_y$ of any fiber $\pi^{-1}(y)$ agrees with the inverse image by
$\pi_r$ of a linear variety of $\Pp^r$ of dimension $r-s-1$, and
hence is of pure dimension $r-s-1$. On the other hand, the fact that
$V\cap \{\lambda_0\cdot X=\cdots=\lambda_{s+1}\cdot X=u_{s+2}\cdot
X=\cdots=u_r\cdot X=0\}$ is empty immediately implies that the set
of exceptional points $V\cap \{\lambda_0\cdot
X=\cdots=\lambda_{s+1}\cdot X=0\}$ of $\pi$ is of pure dimension
$r-s-2$.

As a consequence, defining $\mathcal{H}_2'\subset(\Pp^n)^{s+2}$ as
the zero locus of the polynomial $B\in\cfq[U_0\klk U_{s+1}]$
finishes the proof of the lemma.
\end{proof}

\begin{lemma}\label{lemma: L no corta sing multih}
There exists an hypersurface $\mathcal{H}_2''\subset(\Pp^n)^{s+2}$,
defined by a multihomogeneous polynomial of
$\cfq[\boldsymbol{\Lambda}]$ of multidegree at most
$D^{r-s-1}\delta$ in each group of variables $\Lambda_i$, with the
following property: if
$\boldsymbol{\lambda}\in(\Pp^n)^{s+2}\setminus\mathcal{H}_2''$,
then $\Sigma\cap L$ is empty.
\end{lemma}
\begin{proof}
Arguing as in the proof of Proposition \ref{prop: polar variety as a
subset of a variety of dimension s} we see that there exists a
projective variety $Z\subset\Pp^n$ of pure dimension $s+1$ and
degree at most $D^{r-s-1}\delta$ with $\Sigma\subset Z$. Let
$\mathcal{F}_Z\in\fq[\boldsymbol{\Lambda}]$ be the Chow form of $Z$.
We have that $\mathcal{F}_Z$ is homogeneous in each group of
variables $\Lambda_i$ and satisfies the upper bound
$\deg_{\Lambda_i} \mathcal{F}_Z\le D^{r-s-1}\delta$ for $0\le i\le
s+1$.

Let $\boldsymbol{\lambda}\in(\Pp^n)^{s+2}$ be such that
$\mathcal{F}_Z(\boldsymbol{\lambda})\not=0$ holds and let
$L:=\{\lambda_i\cdot X=0\ (0\le i\le s+1)\}$. Then $Z\cap L$ is
empty and hence so is $\Sigma\cap L$. Therefore, defining
$\mathcal{H}_2''\subset(\Pp^n)^{s+2}$ as the zero locus of
$\mathcal{F}_Z$ finishes the proof of the lemma.
\end{proof}

Similarly to Section \ref{section: polar varieties}, we consider the
following incidence variety:
\begin{eqnarray*}
W_s:=(V_{\rm sm} \times \mathcal{U})\cap\{\Lambda_0\cdot
X=0,\dots,\Lambda_{s+1}\cdot X=0,\qquad\
\\\Delta_1(\boldsymbol{\Lambda},X)
=0,\dots,\Delta_N(\boldsymbol{\Lambda},X)=0\},
\end{eqnarray*}
where $\mathcal{U}\subset(\Pp^n)^{s+2}$ is the Zariski open subset
of $(s+ 2)\times (n+1)$--matrices of maximal rank and
$\Delta_1\klk\Delta_N$ are the maximal minors of the generic version
$\M(X,\boldsymbol{\Lambda})$ of the matrix of (\ref{eq: matrix for
polar variety}).
\begin{proposition}\label{prop: W_s es irreducible multih}
$W_s$ is an irreducible subvariety of $\Pp^n\times \mathcal{U}$ of
dimension $t-1$.
\end{proposition}
\begin{proof}
Let $\pi_1:W_s\to V_{\rm sm}$ be the linear mapping defined by
$\pi_1(x,\boldsymbol{\lambda}):= x$. Fix $x\in V_{\rm sm}$ and
consider the fiber $\pi_1^{-1}(x)$. We have that
$\pi_1^{-1}(x)=\{x\}\times \mathcal{L}$, where $\mathcal{L} \subset
\mathcal{U}$ denotes the set of matrices
$\boldsymbol{\lambda}:=(\lambda_0,\dots,\lambda_{s+1})$ such that
$\lambda_j\cdot x=0$ for $0\le j\le s+1$ and the matrix ${\sf
M}(x,\boldsymbol{\lambda})$ is not of full rank. The latter
condition is equivalent to
\begin{equation}\label{eq: cond var incidencia singularidad multih}
\langle\lambda_0,\dots,\lambda_{s+1}\rangle \cap \big\langle\nabla
F_1(x)\klk\nabla F_{n-r}(x)\big\rangle \neq \{\boldsymbol{0}\} ,
\end{equation}
where $\langle v_0\klk v_m\rangle \subset \A^{n+1}$ is the linear
variety spanned by $v_0\klk v_m$ in $\A^{n+1}$. Let
$\mathbb{V}:=\{v\in\A^{n+1}:v\cdot x=0\}$. Observe that $\nabla
F_j(x)\in\mathbb{V}$ for $1\le j\le n-r$. Then (\ref{eq: cond var
incidencia singularidad multih}) holds if and only if
$\lambda_0\klk\lambda_{s+1}$ are not linearly independent in the
quotient $\cfq$--vector space
$$
\mathbb{W}:=\mathbb{V}/\big(\nabla F_1(x)\klk\nabla F_{n-r}(x)\big).
$$
This shows that $\mathcal{L}$ is, modulo $\big(\nabla
F_1(x)\klk\nabla F_{n-r}(x)\big)$, isomorphic to the Zariski open
set $L_{s+1}'(\A^{s+2},\mathbb{W})\cap\mathcal{U}$ of
$L_{s+1}'(\A^{s+2},\mathbb{W})$, where
$$
L_{s+1}'(\A^{s+2},\mathbb{W}):=\{f\in{\rm
Hom}_{\cfq}(\A^{s+2},\mathbb{V}):{\rm rank}(f)\le s+1\}.
$$
According to \cite[Proposition 1.1]{BrVe88},
$L_{s+1}'(\A^{s+2},\mathbb{W})$ is an irreducible variety of
dimension $(s+1)(r+1)$.  Since we are considering subspaces of
$\mathbb{V}$ of dimension $s+2$ modulo $\langle\nabla
F_1(x)\klk\nabla F_{n-r}(x)\rangle$, which has dimension $n-r$, it
follows that $\pi_1^{-1}(x)=\{x\}\times \mathcal{L}$ is an open
dense subset of an irreducible variety of $\Pp^n\times
\A^{(n+1)(s+2)}$ of dimension
$(s+1)(r+1)+(n-r)(s+2)=(n+1)(s+2)-r-1$, or equivalently, an
irreducible variety of $\Pp^n\times\mathcal{U}$ of dimension
$t-r-1$.

Combining these arguments with the proof of \cite[\S I.6.3, Theorem
8]{Shafarevich94} we conclude that $W_s$ is an irreducible variety
of dimension $t-1$.\end{proof}

An immediate consequence of Proposition \ref{prop: W_s es
irreducible multih} is that the Zariski closure of the image of the
projection $\pi_2:W_s\to\mathcal{U}$ is an irreducible variety of
dimension at most $t-1$. Our next result strengthens somewhat this
conclusion and provides further quantitative information.
\begin{proposition}\label{prop: upper bound degree Hs multih}
Let $\mathcal{H}_s\subset(\Pp^n)^{s+2}$ be the Zariski closure of
the image of $\pi_2:W_s\to\mathcal{U}$. Then $\mathcal{H}_s$ is an
hypersurface of $(\Pp^n)^{s+2}$, defined by a multihomogeneous
polynomial of $\cfq[\boldsymbol{\Lambda}]$ of degree at most $\delta
D^{r-s-2}(D+r-s-1)$ in each group of variables $\Lambda_i$.
\end{proposition}
\begin{proof}
We first prove that $\mathcal{H}_s$ is an hypersurface. For this
purpose, it suffices to show that there exists a zero dimensional
fiber of $\pi_2$, because the theorem on the dimension of fibers
(see, e.g., \cite[\S I.6, Theorem 7]{Shafarevich94}) readily implies
our assertion.

Fix generic linear forms $\lambda_0\cdot X\klk\lambda_s\cdot X$. By
the Bertini theorem in the form of \cite[Proposition 1.3]{GhLa02a}
we have that $V\cap \{\lambda_0\cdot X=\cdots=\lambda_s\cdot X=0\}$
is nonsingular of pure dimension $r-s-1$. Choose
$\lambda_{s+1}\in\Pp^n$ such that $V\cap \{\lambda_0\cdot
X=\cdots=\lambda_{s+1}\cdot X=0\}$ is singular. From \cite[Appendix,
Theorem 2]{Hooley91} it follows that the singular locus of $V\cap
\{\lambda_0\cdot X=\cdots=\lambda_{s+1}\cdot X=0\}$ has dimension
zero. Since such a singular locus is isomorphic to the fiber
$\pi_2^{-1}(\lambda_0\klk\lambda_{s+1})$, we deduce the existence of
a zero--dimensional fiber of $\pi_2$, which completes the proof of
the first assertion.

Next, fix $r-s-1$ generic linear combinations, say
$\Delta^1\klk\Delta^{r-s-1}$, of $\Delta_1(\boldsymbol{\Lambda},X)
\klk \Delta_N(\boldsymbol{\Lambda},X)$, such that the subvariety
$W_s'\subset(\Pp^n)^{s+3}$ defined by the set of common zeros of the
equations
\begin{eqnarray*}
F_1=0,\dots,F_{n-r}=0,\Lambda_0\cdot X=0,\dots,\Lambda_{s+1}\cdot
X=0,\\ \Delta^1(\boldsymbol{\Lambda},X)=0,\dots,
\Delta^{r-s-1}(\boldsymbol{\Lambda},X)=0,\qquad\quad
\end{eqnarray*}
is of pure dimension $t-1$ and let
$\mathcal{H}_s'\subset(\Pp^n)^{s+2}$ be the union of the components
of the Zariski closure of $\pi_2(W_s')$ of dimension $t-1$. Then
$\mathcal{H}_s'$ is an hypersurface containing $\mathcal{H}_s$.

Finally, we estimate the multidegree of $\mathcal{H}_s'$. For this
purpose, we consider the class $[W_s']$ of $W_s'$ in the Chow ring
$\mathcal{A}^*((\Pp^n)^{s+3})$ of $(\Pp^n)^{s+3}$. Denote by
$\theta_{j-2}$ the class of the inverse image of a hyperplane of
$\Pp^n$ under the $j$th canonical projection $(\Pp^n)^{s+3}\to\Pp^n$
for $1\le j\le s+2$. Then the multihomogeneous B\'ezout theorem
(see, e.g., \cite[Theorem 1.11]{DaKrSo13}) asserts that
\begin{eqnarray*}
[W_s']&=&\prod_{i=1}^{n-r}(d_i\theta_{-1})
\prod_{j=0}^{s+1}(\theta_{-1}+\theta_j)\prod_{k=1}^{r-s-1}(D\theta_{-1}+
\theta_0\plp\theta_{s+1})\\
&=&\delta D^{r-s-2}(D+r-s-1)(\theta_{-1})^n(
\theta_0\plp\theta_{s+1})\\&&+\mbox{terms of lower degree in
}\theta_{-1}.
\end{eqnarray*}
On the other hand
$[\mathcal{H}_s']=\deg_X\!\!H_s'\,\theta_{-1}+\deg_{\Lambda_0}\!\!H_s'\,\theta_0
\plp\deg_{\Lambda_{s+1}}\!\!H_s'\,\theta_{s+1}$, where
$H_s'\in\fq[\Lambda]$ is a polynomial of minimal degree defining
$\mathcal{H}_s'$. Let $\jmath:\mathcal{A}^*\big((\Pp^n)^{s+2}\big)
\hookrightarrow \mathcal{}A^*\big((\Pp^n)^{s+3}\big)$ be the
injective $\Z$--map $P\mapsto(\theta_{-1})^nP$ induced by $\pi_2$.
Then \cite[Proposition 1.16]{DaKrSo13} shows that
$\jmath([\mathcal{H}_s'])\le [W_s']$, where the inequality is
understood in a coefficient--wise sense. This implies
$\deg_{\Lambda_j}\!H_s'\le\delta D^{r-s-2}(D+r-s-1)$ for $0\le j\le
s+1$, finishing thus the proof of the proposition. \end{proof}

Finally, combining Lemmas \ref{lemma: dimension de la fibra} and
\ref{lemma: L no corta sing multih} and Proposition \ref{prop: upper
bound degree Hs multih} we obtain the main result of this section.
\begin{corollary}\label{coro: singular locus}
There exists an hypersurface $\mathcal{H}_2\subset(\Pp^n)^{s+2}$,
defined by a multihomogeneous polynomial of degree at most
$\big(D^{r-s-2}(2D+r-s-1)+1\big)\delta$ in each group of variables
$\Lambda_i$, with the following property: if
$\boldsymbol{\lambda}\in(\Pp^n)^{s+2}\setminus\mathcal{H}_2$, then
$V\cap L$ is nonsingular of pure dimension $r-s-2$ and
$\boldsymbol{\lambda}$ satisfies the conditions in the statements of
Lemmas \ref{lemma: dimension de la fibra} and \ref{lemma: L no corta
sing multih}.
\end{corollary}
%
%
%
\section{An effective Bertini theorem}\label{section: Bertini}
This section is devoted to obtain an effective version of the
Bertini smoothness Theorem. The Bertini smoothness theorem (see,
e.g., \cite[II.6.2, Theorem 2]{Shafarevich94}) asserts that, given a
dominant morphism $f:V_1\to V_2$ of irreducible varieties defined
over a field of characteristic zero with $V_1$ nonsingular, there
exists a dense open set $U$ of $V_2$ such that the fiber $f^{-1}(y)$
is nonsingular for every $y\in U$. An effective version of this
result provides an upper bound of the degree of the subvariety of
$V_2$ consisting of the points defining singular fibers. The
effective version we shall obtain holds without any restriction on
the characteristic of the ground field and generalizes significantly
\cite[Theorem 5.3]{CaMa07}.

Let $V \subset \Pp^n$ be an ideal--theoretic complete intersection
defined over $\fq$ of dimension $r$ and degree $\delta$. Let
$F_1\klk F_{n-r} \in \fq[\xon]$ be homogeneous polynomials of
degrees $d_1\ge\cdots\ge d_{n-r}\ge 2$ respectively, which generate
the ideal $I(V)$ of $V$. We assume that the singular locus $\Sigma$
of $V$ has dimension at most $s\le r-2$.

Let $\boldsymbol{\lambda}:=(\lambda_0\klk\lambda_{s+1})\in
(\Pp^n)^{s+2}\setminus\mathcal{H}_2$, where
$\mathcal{H}_2\subset(\Pp^n)^{s+2}$ is the hypersurface of the
statement of Corollary \ref{coro: singular locus}, and let
$Y_j:=\lambda_j\cdot X$ for $0\le j\le s+1$. Let $\pi:V\to\Pp^{s+1}$
be the linear mapping defined by $Y_0\klk Y_{s+1}$. The set of
exceptional points of $\pi$ is equal to $V\cap L$, where
$L:=\{Y_0=\cdots=Y_{s+1}=0\}$.
\begin{remark}\label{remark: polar var intersection L}
With assumptions and notations as above, $\Sigma\cap L$ is empty,
and thus, ${\sf M}(L)\cap L={\rm Sing}(V\cap L)$ is also empty.
\end{remark}
\begin{proof}
The fact that $\boldsymbol{\lambda}\notin\mathcal{H}_2$ implies that
$\boldsymbol{\lambda}\notin\mathcal{H}_2''$, where
$\mathcal{H}_2''\subset(\Pp^n)^{s+2}$ is the hypersurface of the
statement of Lemma \ref{lemma: L no corta sing multih}. Then
$\Sigma\cap L$ is empty.

Observe that
${\rm Sing}(V\cap L)$
is the set of points of $V_{\rm sm}\cap L=V\cap L$ where the
intersection is not transversal. For a point $x\in V\cap L$, the
intersection is not transversal if and only if $\dim(T_xV\cap
L)>r+(n-s-2)-n=r-s-2$. This shows that $x\in \mathrm{Sing}(V\cap L)$
if and only if $x\in{\sf M}(L)\cap L$, namely ${\sf M}(L)\cap L={\rm
Sing}(V\cap L)$.

Finally, $\boldsymbol{\lambda}\notin\mathcal{H}_s$, where
$\mathcal{H}_s\subset(\Pp^n)^{s+2}$ is the hypersurface of the
statement Proposition \ref{prop: upper bound degree Hs multih},
which implies that $V\cap L$ is nonsingular.
\end{proof}

We shall prove that, for a generic choice of $Y_0\klk Y_{s+1}$,
there exists a nonempty open subset $U$ of $\Pp^{s+1}$ such that the
Zariski closure $V_y$ of $\pi^{-1}(y)$ is nonsingular for every
$y\in U$. Furthermore, we shall provide an estimate on the degree of
the generic condition underlying the choice of $Y_0\klk Y_{s+1}$ and
the degree of the variety $\Pp^{s+1}\setminus U$ yielding
nonsingular fibers.

The first step is to obtain a sufficient condition for the
nonsingularity of the linear section $V_y$ of $V$ defined by a point
$y\in \Pp^{s+1}$. Fix $y:=(y_0:\dots:y_s)\in\Pp^{s+1}$ and assume
without loss of generality that $y_0\not=0$ holds. Then
$$V_y=\big\{x\in V:y_jY_0(x)-y_0Y_j(x)=0\ (1\le j\le s)\big\}.$$
In particular, we have that $V\cap L\subset V_y$. Since
$\boldsymbol{\lambda}\notin \mathcal{H}_2$, where
$\mathcal{H}_2\subset(\Pp^n)^{s+2}$ is the hypersurface of Corollary
\ref{coro: singular locus}, it turns out that $\Sigma\cap L$ is
empty. Furthermore, by Remark \ref{remark: polar var intersection L}
we have that $V\cap L$ is nonsingular. This in particular implies
that any point of $V\cap L$ is a nonsingular point of $V_y$.

Now we can state and prove a sufficient condition for the
nonsingularity of the linear section $V_y$ of $V$. For this purpose,
we shall consider as before the linear mapping
$\varphi_x:T_xV\to\Pp^{s+1}$ defined by $Y_0\klk Y_{s+1}$.
\begin{lemma}\label{lemma:fibra no singular}
Let $y$ be a point of $\Pp^{s+1}$ such that for every
$x\in\pi^{-1}(y)$ the following conditions hold:
\begin{enumerate}
  \item[($i$)] $x$ is a regular point of $V$,
  \item[($ii$)] the set of exceptional points of $\varphi_x$ has dimension
  at most $r-s-2$.
\end{enumerate}
Then $V_y$ is a nonsingular variety.
\end{lemma}
\begin{proof}
Since $V_y$ is of pure dimension $r-s-1$, it suffices to prove that
for every $x \in V_y$ the tangent space $T_xV_y$ has dimension at
most $r-s-1$. Fix $x\in \pi^{-1}(y)$. Condition $(i) $ implies that
the tangent space $T_xV$ has dimension $r$. Consider the linear
mapping
     $$
       \begin{array}{rrcl}
       \varphi_{x}|_{T_x{V_y}}:& T_x V_y & \to & \Pp^{s+1}
       \\
        & v  & \mapsto & (Y_0(v):\dots: Y_{s+1}(v)).
       \end{array}
     $$
It is clear that the set $\E_{x,y}$ of exceptional points of
$\varphi_{x}|_{T_x{V_y}}$ is contained in the set $\E_{x}$ of
exceptional points of $\varphi_x$. From the fact that the
restriction $\pi|_{V_y}:V_y\to\Pp^{s+1}$ maps $V_y$ to the point $y$
it follows that the dimension of $\varphi_{x}(T_x{V_y})$ is equal to
0. By the Dimension theorem of linear algebra (see, e.g.,
\cite[Chapter 8, Section 4]{HoPe68a}) we have that
$$\dim T_xV_y = \dim \E_{x,y} + \dim \varphi_{x}(T_x{V_y})+1.$$
From this and condition $(ii)$  we deduce that
$$\dim T_xV_y\le\dim \E_x+1\le r-s-1.$$
We conclude that $\dim T_xV_y=r-s-1$ and therefore $x$ is a regular
point of $V_y$.

Finally, let $x\in V_y\setminus\pi^{-1}(y)$. Then $x\in V\cap L$ and
is a regular point of $V\cap L$. From the fact that $F_1\klk
F_{n-r},Y_0\klk Y_{s+1}$ define the ideal of $V\cap L$ we easily
deduce that $x$ is also a regular point of $V_y$, which finishes the
proof of the lemma. \end{proof}

The critical point of our effective Bertini theorem is the analysis
of the set of regular points $x\in V$ for which the dimension of the
set of exceptional points of $\varphi_x$ has dimension at least
$r-s-1$, where $Y_0\klk Y_{s+1}$ are suitably chosen. Remark
\ref{lemma: polar variety through kernels} asserts that this set is
the polar variety $\M(L)$, where $L:=\{Y_0=\cdots=Y_{s+1}=0\}$.

Now we are ready to state our effective version of the Bertini
smoothness theorem.
\begin{theorem}\label{theorem: Bertini}
Let $\mathcal{H}_1$ and $\mathcal{H}_2$ be the hypersurfaces of
$(\Pp^n)^{s+2}$ of the statements of Theorem \ref{theorem: dimension
polar variety intrinseca} and Corollary \ref{coro: singular locus}
respectively, let $\mathcal{H}:=\mathcal{H}_1\cup\mathcal{H}_2$ and
let $\boldsymbol{\lambda}:=(\lambda_0\klk \lambda_{s+1})
\in(\Pp^n)^{s+2}\setminus\mathcal{H}$. Let $Y_j:=\lambda_j\cdot X$
for $0\le j\le s+1$, let $L:=\{Y_0=\cdots=Y_{s+1}=0\}$ and let
$\pi:V\to\Pp^{s+1}$ be the linear mapping defined by $Y_0\klk
Y_{s+1}$. Then there exists a closed set $W(L)\subset \Pp^{s+1}$ of
dimension at most $s$ and degree at most $D^{r-s}\delta$ such that
for every  $y \in \Pp^{s+1}\setminus W(L)$ the linear section $V_y$
of $V$ is nonsingular of pure dimension $r-s-1$.
\end{theorem}
\begin{proof}
Since $\boldsymbol{\lambda}\notin \mathcal{H}$, by Theorem
\ref{theorem: dimension polar variety intrinseca} it follows that
the polar variety $\M(L)$  has dimension $s$. Let $Z(L)\subset V$ be
the subvariety of dimension $s$ and degree at most $D^{r-s}\delta$
whose existence was established in Proposition \ref{prop: polar
variety as a subset of a variety of dimension s}. Then
      $$
        \M(L) \cup  \Sigma \subset Z(L).
      $$
Defining  $W(L):=\pi(Z(L))$ it turns out that $W(L)\subset
\Pp^{s+1}$ has dimension at most  $s$. Furthermore, by
(\ref{eq:degree linear projection}) we have that $W(L)$ has degree
at most $D^{r-s}\delta$.

Let $y\in\Pp^{s+1}\setminus W(L)$. By Corollary \ref{coro: singular
locus} we have that $V_y$ is of pure dimension $r-s-1$. Furthermore,
the conditions of the statement of Lemma \ref{lemma:fibra no
singular} are satisfied, and hence $V_y$ is nonsingular. This
finishes the proof of the theorem. \end{proof}
\begin{remark}\label{rem: V_y is contained in V_sm}
With notations and assumptions as in Theorem \ref{theorem: Bertini},
for $\boldsymbol{\lambda}\in(\Pp^n)^{s+2}\setminus\mathcal{H}$ and
$y \in \Pp^{s+1}\setminus W(L)$, the linear section $V_y$ is
contained in $V_{\rm sm}$. Indeed, by the choice of $y$ it turns out
that any point $x\in\pi^{-1}(y)$ is a regular point of $V$. On the
other hand, if $x\in V_y\setminus\pi^{-1}(y)$, then $x\in V\cap L$,
and $V\cap L\subset V_{\rm sm}$ by Remark \ref{remark: polar var
intersection L}.
\end{remark}

Since the linear section $V_y$ is a nonsingular projective complete
intersection for $y\notin W(L)$, the Hartshorne connectedness
theorem (see, e.g., \cite[VI, Theorem 4.2]{Kunz85}) shows that $V_y$
is connected, which implies that it is absolutely irreducible.

In what follows, we shall frequently use the notation
$$B_{\boldsymbol{d},s}:=D^{r-s-2}\delta
\big(((n-s)(r-s)+2)D+r-s-1)\big)+\delta+1.$$
\begin{corollary}\label{coro: existencia fibra no singular fq-definible}
For $q>\max\{B_{\boldsymbol{d},s},D^{r-s}\delta\}$, there exists $y
\in \Pp^{s+1}(\fq)$ such that $V_y$ is a nonsingular $\fq$--variety
of pure dimension $r-s-1$. In other words, $V$ has a nonsingular
linear section of pure dimension $r-s-1$ defined over $\fq$.
\end{corollary}
\begin{proof}
According to Theorem \ref{theorem: dimension polar variety
intrinseca} and Corollary \ref{coro: singular locus}, there exist
hypersurfaces $\mathcal{H}_1$ and $\mathcal{H}_2$ of $(\Pp^n)^{s+2}$
such that for any $\boldsymbol{\lambda}
\in(\Pp^n)^{s+2}\setminus(\mathcal{H}_1\cup \mathcal{H}_2)$, the
associated linear space $L$ satisfy the requirements of the
statements of Theorem \ref{theorem: dimension polar variety
intrinseca} and Corollary \ref{coro: singular locus}. Since
$\mathcal{H}_1$ and $\mathcal{H}_2$ are defined by multihomogeneous
polynomials of $\cfq[\boldsymbol{\Lambda}]$ of degree at most
$(n-s)(r-s)D^{r-s-1}\delta+1$ and
$\delta\big(D^{r-s-2}(2D+r-s-1)+1\big)$ in each group of variables
$\Lambda_i$ respectively, it follows that
$\mathcal{H}:=\mathcal{H}_1\cup \mathcal{H}_2$ is defined by a
multihomogeneous polynomial of $\cfq[\boldsymbol{\Lambda}]$ of
degree at most $B_{\boldsymbol{d},s}$ in each group of variables
$\Lambda_i$. By Corollary \ref{coro: existencia no-cero pol multih}
we conclude that, if $q>B_{\boldsymbol{d},s}$, then there exists a
point $\boldsymbol{\lambda}\in
(\Pp^n(\fq))^{s+2}\setminus\mathcal{H}$. Let $\pi:V\to\Pp^{s+1}$ be
the linear mapping defined by the corresponding linear forms
$Y_0\klk Y_{s+2}$. Then Theorem \ref{theorem: Bertini} shows that
there exists an hypersurface $W(L)\subset\Pp^{s+1}$ of degree at
most $D^{r-s}\delta$ such that, for $y\in\Pp^{s+1}\setminus W(L)$,
the linear section $V_y$ is nonsingular. Since $q>D^{r-s}\delta$, we
see that there exists $y\in\Pp^{s+1}(\fq)\setminus W(L)$, from which
the corollary follows. \end{proof}
%
%
\section{Results of existence of smooth $\fq$--rational points}
\label{section: existence}
In this section we obtain results of existence of smooth
$\fq$--rational points of a complete intersection $V \subset \Pp^n$
defined over $\fq$, of dimension $r$, degree $\delta$, multidegree
$\boldsymbol{d}:= (d_1\klk d_{n-r})$ with $d_1\ge \cdots\ge
d_{n-r}\ge 2$ and singular locus $\Sigma$ of dimension at most $s$.
More precisely, we establish conditions on $q$ which imply that
$V_{\rm sm}(\fq)$ is not empty.

The usual approach to this kind of results relies on a combination
of the available estimates on the number of $\fq$--rational points
and upper bounds for the number of singular $\fq$--rational points.
Instead of doing this, we shall use the effective version of the
Bertini smoothness theorem of Section \ref{section: Bertini} in
order to establish the existence of a nonsingular linear section of
$V$ defined over $\fq$. Such a singular section will be contained in
$V_{\rm sm}$. We shall combine this result with the following
well-known estimate on the number of $\fq$--rational points of a
nonsingular complete intersection $W\subset \Pp^n$ defined over
$\fq$, of dimension $r$, degree $\delta$ and multidegree
$\boldsymbol{d}$ due to P. Deligne \cite{Deligne74}:
\begin{equation} \label{eq: estimate Deligne}
\big||W(\fq)|- p_r\big|\leq b_r'(n,\boldsymbol{d})\,q^{r/2},
\end{equation}
where $b_r'(n,\boldsymbol{d})$ denotes the $r$th primitive Betti
number of any nonsingular complete intersection of $\Pp^n$ of
dimension $r$ and multidegree $\boldsymbol{d}$.


In the sequel we shall frequently use the following explicit
expressions for $b_r'(n,\boldsymbol{d})$ with $r\in\{1,2\}$ (see,
e.g., \cite[Theorem 4.1]{GhLa02a}):
\begin{eqnarray*}
b_1'(n,\boldsymbol{d})&\!\!\!\!=&\!\!\!\!(d_1\cdots d_{n-1})(d_1\plp d_{n-1}-n-1)+2,\\
b_2'(n,\boldsymbol{d})&\!\!\!\!=&\!\!\!\!(d_1\cdots
d_{n-2})\left(\!\!\binom{n+1}{2}- (n+1)\!\!\sum_{1\le i\le
n-2}\!\!\!d_i+\!\!\sum_{1\le i\le j\le n-2}\!\!\!d_id_j\!\right)-3.
\end{eqnarray*}
\begin{remark}\label{remark: upper bound b_2}
Let $V\subset\Pp^n$ be a nonsingular complete intersection defined
over $\fq$, of dimension 2 and multidegree $\boldsymbol{d}:=(d_1\klk
d_{n-2})$. Let $D:=\sum_{i=1}^{n-2}(d_i-1)$. Observe that $\deg
V=d_1\cdots d_{n-2}$ and the following upper bound holds:
\begin{equation}\label{eq: upper bound b_2}
b_2'(n,\boldsymbol{d})\le (n-1)D^2\deg V.
\end{equation}
Indeed, we have
$$
-(n+1)\!\!\sum_{1\le i\le n-2}\!\!d_i+\!\!\sum_{1\le i\le j\le
n-2}\!\!d_id_j\le\sum_{i=1}^{n-2}
d_i\Bigg(\sum_{i=1}^{n-2}d_i-n-1\Bigg)= \sum_{i=1}^{n-2}d_i(D-3).
$$
Using the inequality $\sum_{i=1}^{n-2}d_i\le (n-1)D$, we obtain
$$b_2'(n,\boldsymbol{d})\le \deg V\left(\binom{n+1}{2}+(n-1)D(D-3)\right)
\le (n-1)D^2\deg V.$$
This shows (\ref{eq: upper bound b_2}). \end{remark}

Let $B_{\boldsymbol{d},s}:=
D^{r-s-2}\delta\big(((n-s)(r-s)+2)D+r-s-1)\big)+\delta+1$. According
to Corollary \ref{coro: existencia fibra no singular fq-definible},
if $q>\max\{B_{\boldsymbol{d},s},D^{r-s}\delta\}$ then there exists
a nonsingular linear section $S$ of $V$ defined over $\fq$ of pure
dimension $r-s-1$ contained in $V_{\rm sm}$. We are going to prove
that the number of $\fq$--rational points in $S$ is strictly
positive, showing thus that $V$ has smooth $\fq$--rational points.
We have the following result.
\begin{theorem}\label{theorem: existence}
Let $V\subset \Pp^n$ be a complete intersection defined over $\fq$,
of dimension $r\ge 2$, degree $\delta$, multidegree $\boldsymbol{d}$
and singular locus $\Sigma$ of dimension at most $s\le r-2$. If  $
q>\max\Big\{B_{\boldsymbol{d},s},D^{r-s}\delta,
\big(b_{r-s-1}'(n-s-1,\boldsymbol{d})\big)^{2/(r-s-1)}\Big\}$, then
$V$ has a smooth $\fq$--rational point.
 \end{theorem}
  \begin{proof}
Let $S$ be the nonsingular linear section of $V$ whose existence is
assured by Corollary \ref{coro: existencia fibra no singular
fq-definible}. Since $S$ is a nonsingular complete intersection
defined over $\fq$ of dimension $r-s-1$, from  (\ref{eq: estimate
Deligne}) it follows that
$$ |S(\fq)|\ge
p_{r-s-1}-b'q^{\frac{r-s-1}{2}}>q^{\frac{r-s-1}{2}}
\big(q^{\frac{r-s-1}{2}}-b'\big), $$
where $b':=b_{r-s-1}'(n-s-1,\boldsymbol{d})$. Our conditions
immediately implies that the right--hand side in the previous
expression is positive. Furthermore, according to Remark \ref{rem:
V_y is contained in V_sm}, we have that $S\subset V_{\rm sm}$,
finishing the proof of the theorem. \end{proof}

Next we discuss two particular instances of this result.

\begin{corollary}\label{coro: existence codim 2}
With notations and assumptions as in Theorem \ref{theorem:
existence}, if
$$q>\left\{\!\!\begin{array}{cl}
\big(\delta(D-2)+2\big)^2,&\textit{for}\ D\ge 5\ \textit{or}\ D=4\
\textit{and}\ n-r>1,\\
\big(2(n-r+3)D+2\big)\delta+1,\!\!&\textit{otherwise},
\end{array}\right.$$
then $V$ has a smooth $\fq$--rational point.
\end{corollary}
\begin{proof}
Observe that $b_1'(n-r+1,\boldsymbol{d})=\delta(D-2)+2$. Therefore,
applying Theorem \ref{theorem: existence} with $s=r-2$, we conclude
that, if
\begin{equation}\label{eq: coro existence points r-2}
q>\max\Big\{\big(2(n-r+3)D+2\big)\delta+1,D^2\delta,
\big(\delta(D-2)+2\big)^2\Big\},
\end{equation}
then $V$ has a smooth $\fq$--rational point. For $D\le 2$ we have
$D^2\delta\le \big(2(n-r+3)D+2\big)\delta+1$, while
$D^2\delta\le(\delta(D-2)+2)^2$ for $D\ge 3$. As a consequence, we
see that (\ref{eq: coro existence points r-2}) is equivalent to
\begin{equation}\label{eq: coro existence points r-2 bis}
q>\max\Big\{\big(2(n-r+3)D+2\big)\delta+1,
\big(\delta(D-2)+2\big)^2\Big\}.
\end{equation}

If $D\ge 6$, then we have
$$
\big(\delta(D-2)+2\big)^2\ge \big(2(D+3)D+2\big)\delta+1\ge
\big(2(n-r+3)D+2\big)\delta+1.
$$
Combining this inequality with (\ref{eq: coro existence points r-2
bis}) and elementary calculations we deduce the statement of the
corollary.

%
\end{proof}
\begin{corollary}\label{coro: existence codim 3}
Let notations and assumptions be as in Theorem \ref{theorem:
existence}. Suppose further that the singular locus of $V$ has
dimension at most $r-3\ge 0$. If $q>3D(D+2)^2\delta$, then $V$ has a
smooth $\fq$--rational point.
\end{corollary}
\begin{proof}
We apply Theorem \ref{theorem: existence} with $s=r-3$. According to
Remark \ref{remark: upper bound b_2}, we have
$b_2'(n-r+2,\boldsymbol{d})\le(n-r+1)D^2\delta$. Therefore, Theorem
\ref{theorem: existence} shows that a sufficient condition for the
existence of a smooth $\fq$--rational point of $V$ is
\begin{equation}\label{eq: aux proof coro existence codim 3}
q>\max\{D^3\delta,D\delta\big((3(n-r+3)+2)D+2\big)+\delta+1\}.
\end{equation}
%
Using the inequality $n-r\le D$, we deduce that
$$
D\delta\big((3(n-r+3)+2)D+2\big)+\delta+1 \le 3D(D+2)^2\delta,
$$
which immediately implies the statement of the corollary.\end{proof}

\section{Estimates on the number of $\fq$--rational points}
\label{section:estimacion}
In this section we obtain estimates on $|V(\fq)|$ for a complete
intersection $V \subset \Pp^n$ of dimension $r$, degree $\delta$ and
multidegree $\boldsymbol{d}:=(d_1\klk d_{n-r})$ with $d_1\ge\cdots
\ge d_{n-r}\ge 2$ for which the singular locus has codimension at
least 2 or 3.

Fix $s\in\{r-2,r-3\}$. Let $D:=\sum_{i=1}^{n-r}(d_i-1)$. Then
Theorem \ref{theorem: dimension polar variety intrinseca} and
Corollary \ref{coro: singular locus} show that there exists an
hypersurface $\mathcal{H}:=\mathcal{H}_1\cup
\mathcal{H}_2\subset(\Pp^n)^{s+2}$, defined by a multihomogeneous
polynomial of $\cfq[\boldsymbol{\Lambda}]$ of degree at most
$$B_{\boldsymbol{d},s}:=D^{r-s-2}\delta\big(((n-s)(r-s)+2)D+r-s-1\big)+\delta+1$$
in each group of variables $\Lambda_i$, with the following property:
for any $\boldsymbol{\lambda}:=(\lambda_0\klk
\lambda_{s+1})\in(\Pp^n)^{s+2}\setminus\mathcal{H}$, let
$Y_j:=\lambda_j\cdot X$ for $0\le j\le s+1$, let $\pi:V\to\Pp^{s+1}$
be the linear mapping defined by $Y_0\klk Y_{s+1}$ and let
$L:=\{Y_0=\cdots=Y_{s+1}=0\}\subset\Pp^n$. Then the following
conditions hold:
\begin{itemize}
  \item[$(i)$] the polar variety $\M(L)$ has dimension $s$,
  \item[$(ii)$] any fiber $\pi^{-1}(y)$ is of pure dimension
  $r-s-1$,
  \item[$(iii)$] the set of exceptional points of $\pi$ is nonsingular of pure dimension
  $r-s-2$.
\end{itemize}
For such a matrix $\boldsymbol{\lambda}$, our effective version of
the Bertini smoothness theorem (Theorem \ref{theorem: Bertini})
asserts that there exists a variety $W_L:=W(L)\subset \Pp^{s+1}$ of
dimension at most $s$ and degree at most $D^{r-s}\delta$ such that
for every $y\in\Pp^{s+1}\setminus W_L$, the fiber $\pi^{-1}(y)$ is a
nonsingular complete intersection. Since $\pi^{-1}(y)$ is
$\fq$--definable for every $y\in\Pp^{s+1}(\fq)$, we can estimate the
number $N_y:=|V_y(\fq)|$ for $y\in\Pp^{s+1}(\fq)\setminus W_L$ using
Deligne's estimate (\ref{eq: estimate Deligne}). On the other hand,
fibers of points in $W_L$ do not make a significant contribution to
the asymptotic behavior of the number of rational points of $V$.
More precisely, we have the following result.
\begin{theorem}\label{theorem: estimate for general s}
Let $V\subset \Pp^n$ be a complete intersection defined over $\fq$,
of dimension $r\ge 2$, degree $\delta$, multidegree $\boldsymbol{d}$
and singular locus of dimension at most $s\in\{r-2,r-3\}$. Then the
following estimate holds:
$$\big||V(\fq)|-p_r\big|\leq b_{r-s-1}'q^{(r+s+1)/2}+
A(n,s,\boldsymbol{d})\,q^{r-1},$$
where $A(n,s,\boldsymbol{d}):=2b_{r-s-1}'+
2\big(7D^{r-s}\delta+1\big)(\delta-1)$ and
$b_{r-s-1}':=b_{r-s-1}'(n-s-1,\boldsymbol{d})$ is the $(r-s-1)$th
primitive Betti number of any nonsingular complete intersection of
$\Pp^{n-s-1}$ of dimension $r-s-1$ and multidegree $\boldsymbol{d}$.
\end{theorem}
\begin{proof}
First we observe that, if $D=1$, then $V$ is a quadric, and the
statement of the theorem follows from known results on the number of
rational points of quadrics (see, e.g., \cite[Theorem 2E]{Schmidt76}
or \cite[Section 6.2]{LiNi83}).

Next we claim that we may assume without loss of generality that
$q>B_{\boldsymbol{d},s}$ holds. Indeed, suppose that $q\le
B_{\boldsymbol{d},s}$ holds. For $D=2$ we have that $V$ is either a
cubic hypersurface or an intersection of two quadrics. In both cases
we have the upper bound $|V(\fq)|\le \delta q^r+p_{r-1}$ (see
\cite{Serre91} and \cite{EdLiXi09}), which implies
$\big||V(\fq)|-p_r\big|\le(\delta-1)q^r\le
B_{\boldsymbol{d},s}(\delta-1)q^{r-1}$. Then the inequality
$B_{\boldsymbol{d},s}\le 10\cdot 2^{r-s}\cdot\delta+1$ completes the
proof of the theorem in this case.

On the other hand, for $D\ge 3$, according to (\ref{eq: projective
upper bound}) we have $|V(\fq)|\le \delta p_r$, and therefore
$\big||V(\fq)|-p_r\big|\le(\delta-1)p_r\le
2B_{\boldsymbol{d},s}(\delta-1)q^{r-1}$. As a consequence, from the
inequality $B_{\boldsymbol{d},s}\le 7D^{r-s}\delta+1$ we easily
deduce the statement of the theorem in this case. This finishes the
proof of our claim.

For $q>B_{\boldsymbol{d},s}$, combining Theorem \ref{theorem:
dimension polar variety intrinseca} and Corollary \ref{coro:
singular locus} with Corollary \ref{coro: existencia no-cero pol
multih} we deduce that there exists $\boldsymbol{\lambda}
\in(\Pp^n(\fq))^{s+2}$ such that conditions $(i)$--$(iii)$ above are
satisfied.

Let $V_y$ be the linear section of $V$ which is obtained as the
Zariski closure of the fiber $\pi^{-1}(y)$ of an arbitrary point
$y\in\Pp^{s+1}$. We express $|V(\fq)|$ in terms of the quantities
$N_y:=|V_y(\fq)|$ with $y\in\Pp^{s+1}(\fq)$:
\begin{equation}\label{eq:auxiliar estimate 1 dim s}|V(\fq)|\ =
\sum_{y\in \Pp^{s+1}(\fq)}( N_y-e)+e= \sum_{y\in \Pp^{s+1}(\fq)}
N_y-(p_{s+1}-1)e,
\end{equation}
where $e:=|(V\cap L)(\fq)|$. Since $V\cap L$ has dimension $r-s-2$,
we have that $e\le\delta p_{r-s-2}$, and thus
$|e-p_{r-s-2}|\le(\delta-1)p_{r-s-2}$, holds.

Subtracting $p_r$ at both sides of (\ref{eq:auxiliar estimate 1 dim
s}) and taking into account the identity
$p_r=p_{s+1}p_{r-s-1}-(p_{s+1}-1)p_{r-s-2}$, we obtain:
\begin{eqnarray}\big|{|V(\fq)|} -
p_r\big|&\le & \sum_{y\in \Pp^{s+1}(\fq)} \left |N_y -
p_{r-s-1} \right| \, + \,(p_{s+1}-1)(\delta-1)p_{r-s-2}\nonumber\\
&\le&
 \sum_{y\in \Pp^{s+1}(\fq)}  \left |N_y -
p_{r-s-1} \right| \, + \, 2(\delta-1)q^{r-1}. \label{eq:auxiliar
estimate 2 dim s}
\end{eqnarray}

Let  $W_L\subset\Pp^{s+1}$ be the variety of the statement of
Theorem \ref{theorem: Bertini}. We can decompose the first term of
the right--hand side of (\ref{eq:auxiliar estimate 2 dim s}) as:
$$\sum_{y\in \Pp^{s+1}(\fq)} \left |N_y - p_{r-s-1} \right|=
\sum_{y\notin  W_L(\fq)} \left |N_y - p_{r-s-1} \right|+ \sum_{y\in
W_L(\fq)} \left |N_y - p_{r-s-1}\right|.$$

In order to estimate the first term in the right--hand side of the
above expression, Theorem \ref{theorem: Bertini} asserts that, if $y
\notin W_L(\fq)$, then $V_y$ is a nonsingular complete intersection
of $\Pp^{n-s-1}$ defined over $\fq$, of pure dimension $r-s-1$,
degree $\delta$ and multidegree $\boldsymbol{d}$. By (\ref{eq:
estimate Deligne}) we have $|N_y - p_{r-s-1}|\leq
b_{r-s-1}'q^{(r-s-1)/2}$. Therefore, we obtain
\begin{equation} \label{eq:sumafibrasregulares dim s}
\sum_{y\notin W_L(\fq)}\left |N_y - p_{r-s-1} \right|\leq
b_{r-s-1}'q^{\frac{r-s-1}{2}}p_{s+1}\le
b_{r-s-1}'q^{\frac{r+s+1}{2}} + 2b_{r-s-1}'q^{r-1}.
\end{equation}

On the other hand, for $y \in W_L(\fq)$ we have $N_y\le \delta
p_{r-s-1}$. Hence, taking into account that $\delta\ge 2$ holds, we
obtain $|N_y-p_{r-s-1}|\leq(\delta-1) p_{r-s-1}$. From (\ref{eq:
projective upper bound}) it follows that $|W_L(\fq)|\le\deg W_L\cdot
p_s$ holds and thus
\begin{equation}
\label{eq:sumafibrassingulares dim s} \sum_{y\in W_L(\fq)}\left |N_y
- p_{r-s-1} \right|\leq (\delta-1) p_{r-s-1}\cdot \deg W_L\cdot
p_s\leq 4(\delta-1) \deg W_L\cdot q^{r-1} .\end{equation}

Combining (\ref{eq:auxiliar estimate 2 dim s}),
(\ref{eq:sumafibrasregulares dim s}), (\ref{eq:sumafibrassingulares
dim s}), we conclude that
$$
\big||V(\fq)|-p_r\big|\leq
b_{r-s-1}'q^{\frac{r+s+1}{2}}+2\left(b_{r-s-1}'+
(2D^{r-s}\delta+1)(\delta-1)\right)q^{r-1}.
$$
From this estimate we easily deduce the statement of the theorem.
\end{proof}

Our next result is concerned with the number of smooth
$\fq$--rational points.
\begin{theorem}\label{theorem: estimate for general s smooth}
Let notations and assumptions be as in Theorem \ref{theorem:
estimate for general s}. Then the following estimate holds:
$$\big||V_{\rm sm}(\fq)|-p_r\big|\leq b_{r-s-1}'q^{\frac{r+s+1}{2}}+
B(n,s,\boldsymbol{d})\,q^{r-1},$$
where $B(n,s,\boldsymbol{d}):=2b_{r-s-1}'+
2\big(2D^{r-s}\delta+1\big)(\delta-1)+2(s+2)(\delta-1)B_{\boldsymbol{d},s}$.
\end{theorem}
\begin{proof}
Let $\mathcal{H}_1\subset(\Pp^n)^{s+2}$ and
$\mathcal{H}_2\subset(\Pp^n)^{s+2}$ be the hypersurfaces of Theorem
\ref{theorem: dimension polar variety intrinseca} and Corollary
\ref{coro: singular locus} respectively, and let
$\mathcal{H}:=\mathcal{H}_1\cup \mathcal{H}_2\subset(\Pp^n)^{s+2}$.
Recall that $\mathcal{H}$ is defined by a multihomogeneous
polynomial of $\cfq[\boldsymbol{\Lambda}]$ of degree at most
$B_{\boldsymbol{d},s}$ in each group of variables $\Lambda_i$. We
have
\begin{align*}
\big||&V_{\rm sm}(\fq)|-p_r\big|\\&=\frac{1}{p_n^{s+2}}
\left(\sum_{\boldsymbol{\lambda}\in((\Pp^n)^{s+2}\setminus\mathcal{H})(\fq)}
\big||V_{\rm sm}(\fq)|-p_r\big|+
\sum_{\boldsymbol{\lambda}\in\mathcal{H}(\fq)}
\big||V_{\rm sm}(\fq)|-p_r\big|\right)\\
&\le\frac{1}{p_n^{s+2}}
\left(\sum_{\boldsymbol{\lambda}\in((\Pp^n)^{s+2}\setminus\mathcal{H})(\fq)}
\big||V_{\rm sm}(\fq)|-p_r\big|+
|\mathcal{H}(\fq)|(\delta-1)p_r\right).\quad \
\end{align*}
By (\ref{eq: upper bound Serre multih}) it follows that
$|\mathcal{H}(\fq)|\le p_n^{s+2}-(q^n-\min\{q,B_{\mathbf{d},s}\}
q^{n-1})^{s+2}$. Hence,
$$\frac{|\mathcal{H}(\fq)|}{(p_n)^{s+2}}
(\delta-1)p_r\le 2(s+2)(\delta-1)B_{\mathbf{d},s}q^{r-1}.$$

For each
$\boldsymbol{\lambda}\in((\Pp^n)^{s+2}\setminus\mathcal{H})(\fq)$,
Theorem \ref{theorem: Bertini} shows that there exists a variety
$W_L\subset \Pp^{s+1}$ of dimension at most $s$ and degree at most
$D^{r-s}\delta$ such that for every $y\in\Pp^{s+1}\setminus W_L$,
the Zariski closure $V_y$ of the fiber $\pi^{-1}(y)$ is a
nonsingular complete intersection contained in $V_{\rm sm}$. Then,
arguing as in the proof of Theorem \ref{theorem: estimate for
general s}, we obtain
\begin{align*}\frac{1}{p_n^{s+2}}
\sum_{\boldsymbol{\lambda}\notin\mathcal{H}(\fq)}&\big||V_{\rm
sm}(\fq)|-p_r\big|\\&\leq
b_{r-s-1}'q^{\frac{r+s+1}{2}}\!+2\left(b_{r-s-1}'+
(2D^{r-s}\delta+1)(\delta-1)\right)q^{r-1}.\end{align*}
From this inequality we easily deduce the statement of the theorem.
\end{proof}
%
%
\subsection{An estimate for a normal complete intersection}
\label{section:estimacion codim 2}
In this section we consider the case $s:=r-2$ of Theorems
\ref{theorem: estimate for general s} and \ref{theorem: estimate for
general s smooth}. We have the following result.
\begin{corollary}\label{coro: estimate normal variety}
Let $V\subset \Pp^n$ be a normal complete intersection defined over
$\fq$, of dimension $r\ge 2$, degree $\delta$ and multidegree
$\boldsymbol{d}$. Then we have
\begin{eqnarray}\label{eq: estimate CMP normal var}
\big||V(\fq)|-p_r\big|&\leq&(\delta(D-2)+2)q^{r-1/2}
+14D^2\delta^2q^{r-1},\\\label{eq: estimate CMP normal var smooth}
\big||V_{\rm sm}(\fq)|-p_r\big|&\leq&(\delta(D-2)+2)q^{r-1/2}
+11(r+1)D^2\delta^2q^{r-1}.
\end{eqnarray}
\end{corollary}
\begin{proof} Applying Theorems \ref{theorem: estimate for general
s} and \ref{theorem: estimate for general s smooth} with $s=r-2$, we
obtain
\begin{eqnarray*}
\big||V(\fq)|-p_r\big|&\leq& b_1'q^{r-1/2}+
A(n,r-2,\boldsymbol{d})\,q^{r-1},\\
\big||V_{\rm sm}(\fq)|-p_r\big|&\leq& b_1'q^{r-1/2}+
B(n,r-2,\boldsymbol{d})\,q^{r-1},\end{eqnarray*}
where $b_1':=b_1'(n-r+1,\boldsymbol{d})$,
\begin{eqnarray*}
A(n,r-2,\boldsymbol{d})&:=&2b_1'+ 2\big(7D^2\delta+1\big)(\delta-1),\\
B(n,r-2,\boldsymbol{d})&:=&2b_1'+
2\big(2D^2\delta+1\big)(\delta-1)+2r(\delta-1)B_{\mathbf{d},r-2}.
\end{eqnarray*}
Taking into account the identity $b_1'=\delta(D-2)+2$ we easily
deduce (\ref{eq: estimate CMP normal var}). On the other hand, using
the inequality $n-r\le D$ we readily obtain (\ref{eq: estimate CMP
normal var smooth}). \end{proof}

For a normal complete intersection $V$ as in Corollary \ref{coro:
estimate normal variety}, we have the following estimate (see
\cite[Corollary 6.2]{GhLa02a}):
\begin{equation}\label{eq: estimate GL normal var for estimate}
\big||V(\fq)|-p_r\big|\le (\delta(D-2)+2)q^{r-1/2}+9\cdot
2^{n-r}((n-r)d+3)^{n+1}q^{r-1},
\end{equation}
where $d:=\max_{1\le i\le n-r}d_i$. On the other hand, for
$q>2(n-r)d\delta+1$, the following estimate holds (see
\cite[Corollary 6.2]{CaMa07}):
\begin{equation}\label{eq: estimate CM normal var for estimate}
\big||V(\fq)|-p_r\big|\le (\delta(D-2)+2)q^{r-1/2}+2
\big((n-r)d\delta\big)^2q^{r-1}.
\end{equation}
These are the most accurate estimates for normal complete
intersections to the best of our knowledge.

The right--hand sides of (\ref{eq: estimate CMP normal var}),
(\ref{eq: estimate GL normal var for estimate}) and (\ref{eq:
estimate CM normal var for estimate}) have the same first terms and
different second terms. For the sake of comparison, we observe that
\begin{eqnarray*}
2^{n-r}((n-r)d+3)^{n+1}\!&\!\!\ge&\!
\big(2(n-r)\big)^{n-r}\left(\sum_{i=1}^{n-r}\frac{d_i}{n-r}\right)^{n-r}
\left(\sum_{i=1}^{n-r}d_i\right)^{r+1}\\
&\ge& \big(2(n-r)\big)^{n-r}\prod_{i=1}^{n-r}d_i
\left(\sum_{i=1}^{n-r}d_i\right)^{r+1}\\
&\ge& \big(2(n-r)\big)^{n-r}D^2\delta
\left(\sum_{i=1}^{n-r}d_i\right)^{r-1},
\end{eqnarray*}
where the mid inequality is a consequence of the AM--GM inequality.

From the previous inequalities we draw several conclusions. First of
all, for varieties of high dimension, say $r\ge (n+1)/2$, (\ref{eq:
estimate CMP normal var}) and (\ref{eq: estimate CM normal var for
estimate}) are clearly preferable to (\ref{eq: estimate GL normal
var for estimate}). In particular, for hypersurfaces the second term
in the right--hand side of both (\ref{eq: estimate CMP normal var})
and (\ref{eq: estimate CM normal var for estimate}) is roughly
quartic in $\delta$ while the one (\ref{eq: estimate GL normal var
for estimate}) contains an exponential term $\delta^{n+1}$. On the
other hand, for varieties of low dimension the second term in the
right--hand side of (\ref{eq: estimate GL normal var for estimate})
might be preferable to (\ref{eq: estimate CMP normal var}) and
(\ref{eq: estimate CM normal var for estimate}). In particular, for
curves the former is roughly linear in $\delta$ while the latter is
quadratic in $\delta$. In this sense, we may say that (\ref{eq:
estimate CMP normal var})--(\ref{eq: estimate CM normal var for
estimate}) somewhat complement (\ref{eq: estimate GL normal var for
estimate}). Finally, we observe that the right--hand side of
(\ref{eq: estimate CM normal var for estimate}) is slightly lower
than that of (\ref{eq: estimate CMP normal var}) but holds only for
$q>2(n-r)d\delta+1$, while (\ref{eq: estimate CM normal var for
estimate}) holds without any restriction on $q$.
%
%
\subsection{An estimate for a complete intersection regular in
codimension 2}
\label{section:estimacion codim 3}
In this section we consider the case of a complete intersection
which is regular in codimension 2, namely $s\le r-3$. We have the
following result.
\begin{corollary}\label{coro: estimate codim 3}
Let $V\subset\Pp^n$ be a complete intersection defined over $\fq$,
of dimension $r\ge 3$, degree $\delta$ and multidegree
$\boldsymbol{d}$ for which the singular locus has dimension at most
$r-3$. Then we have the following estimates:
\begin{eqnarray}\label{eq: estimate CMP codim 3}
\big||V(\fq)|-p_r\big|&\le& 14D^3\delta^2q^{r-1},\\
\big||V_{\rm sm}(\fq)|-p_r\big|&\le& (34r-20)D^3\delta^2q^{r-1}.
\end{eqnarray}
\end{corollary}
\begin{proof}
By Theorems \ref{theorem: estimate for general s} and \ref{theorem:
estimate for general s smooth} it follows that
\begin{eqnarray*}
\big||V(\fq)|-p_r\big|&\leq& A(n,r-3,\boldsymbol{d})q^{r-1},\\
\big||V_{\rm sm}(\fq)|-p_r\big|&\leq&
B(n,r-3,\boldsymbol{d})q^{r-1},
\end{eqnarray*}
where
\begin{eqnarray*}
A(n,r-3,\boldsymbol{d})&:=&3b_2'+ 2\big(7D^3\delta+1\big)(\delta-1),\\
B(n,r-3,\boldsymbol{d})&:=&3b_2'+
2\big(2D^3\delta+1\big)(\delta-1)+2(r-1)(\delta-1)B_{\mathbf{d},r-3},
\end{eqnarray*}
and $b_2':=b_2'(n-r+2,\boldsymbol{d})$. According to Remark
\ref{remark: upper bound b_2} we have $b_2'\le (n-r+1)D^2\delta\le
(D+1)D^2\delta$. Therefore, a simple calculation shows the statement
of the corollary. \end{proof}

With hypothesis as in the statement of Corollary \ref{coro: estimate
codim 3}, the following estimate holds (\cite[Theorem
6.1]{GhLa02a}):
\begin{equation}\label{eq: estimate GL codim3 for estimate}
\big||V(\fq)|-p_r\big|\le b_2'(n-r+2,\boldsymbol{d})q^{r-1}+9\cdot
2^{n-r}\cdot((n-r)d+3)^{n+1}q^{r-3/2}.
\end{equation}
In the comparison of (\ref{eq: estimate CMP codim 3}) and (\ref{eq:
estimate GL codim3 for estimate}) similar remarks can be made as in
the case of normal complete intersections: for high--dimensional
varieties (\ref{eq: estimate CMP codim 3}) might be preferable,
while for low--dimensional varieties (\ref{eq: estimate GL codim3
for estimate}) is likely to be better. Nevertheless, the presence of
exponentials in the second term of the right--hand side of (\ref{eq:
estimate GL codim3 for estimate}) may difficult the application of
such an estimate even for low--dimensional varieties. As an
illustration of this phenomenon, we sketch an application which
requires an estimate like (\ref{eq: estimate CMP codim 3}).
%
%
\subsubsection{The average value set of ``small'' families of polynomials}
Let $T$ be an indeterminate over $\fq$ and let $f\in\fq[T]$. We
define the value set $N(f)$ of $f$ as $N(f):=|\{f(c):c\in\fq\}|$
(cf. \cite{LiNi83}). Birch and Swinnerton--Dyer established the
following significant result \cite{BiSD59}: if $f\in\fq[T]$ with
$\deg(f)=d\ge 1$ is a generic polynomial, then
$$N(f)=\mu_d\,q+
\mathcal{O}(q^{1/2}),$$
where $\mu_d:=\sum_{r=1}^d{(-1)^{r-1}}/{r!}$ and the constant
underlying the $\mathcal{O}$--notation depends only on $d$. Results
on the average value $N(d,0)$ of $N(f)$ when $f$ ranges over all
monic polynomials in $\fq[T]$ with $f(0)=0$ of fixed degree were
obtained by Uchiyama \cite{Uchiyama55b} and improved by Cohen
\cite{Cohen73}. More precisely, in \cite[\S 2]{Cohen73} it is shown
that
$$N(d,0)=\sum_{r=1}^d(-1)^{r-1}\binom{q}{r}q^{1-r}=\mu_d\,q
+\mathcal{O}(1).$$
However, if some of the coefficients of $f$ are fixed, the results
on the average value of $N(f)$ are less precise. In fact, Uchiyama
\cite{Uchiyama55b} and Cohen \cite{Cohen72} obtain the result that
we now state. Let be given $s$ with $0\le s\le d-2$ and elements
$a_{d-1}\klk a_{d-s}\in\fq$. Let $\boldsymbol{a}^s:=(a_{d-s-1}\klk
a_1)$ and let $f_{\boldsymbol{a}^s}:=T^d+\sum_{i=1}^{d-1}a_iT^i$.
Then for $p:=\mathrm{char}(\fq)>d$,
\begin{equation}\label{eq: average value set}
N(d,s)=\!N(d,s;a_{d-1}\klk
a_{d-s})\!:=\frac{1}{q^{d-s-1}}\!\!\!\!\!\!\sum_{\boldsymbol{a}^s\in\fq^{d-s-1}}
\!\!\!\!\!\!N(f_{\boldsymbol{a}^s})= \mu_d\,q+\mathcal{O}(q^{1/2}),
\end{equation}
the constant underlying the $\mathcal{O}$--notation depends only on
$d$ and $s$. Our aim is to obtain an strengthened explicit version
of (\ref{eq: average value set}) which holds without any restriction
on $p$.

For this purpose, we sketch our approach, which relies essentially
on Corollary \ref{coro: estimate codim 3} (proofs will appear
elsewhere). We start with the following identity:
\begin{equation}\label{eq: our formula for value sets}
N(d,s)=\sum_{r=1}^{d-s}(-1)^{r-1}\binom{q}{r}q^{1-r}+
\frac{1}{q^{d-s-1}}\sum_{r=d-s+1}^{d}(-1)^{r-1}\chi(d,s,r),
\end{equation}
where $\chi(d,s,r)$ denotes the number of subsets $\mathcal{X}_r$ of
$\fq$ with exactly $r$ elements such that
$f:=T^d+\sum_{i=d-s}^{d-1}a_iT^i$ admits an interpolant $g\in\fq[T]$
of degree at most $d-s-1$, namely $f|_{\mathcal{X}_r}
=g|_{\mathcal{X}_r}$. A subset $\mathcal{X}_r\subset\fq$ with this
property is called {\em allowable} for $f$.

In this way, we see that the asymptotic behavior of $N(d,s)$ is
determined by that of $\chi(d,s,r)$ for $d-s+1\le r\le d$.
Concerning the latter, we have the following result.
\begin{theorem}\label{theorem: reduction value sets to counting}
Fix $r$ with $d-s+1\le r\le d$ and let $X_1\klk X_r$ be
indeterminates over $\fq$. Suppose that $1\le s\le d/2$. Then there
exists polynomials $R_{d-s}\klk R_{r-1}\in\fq[X_1\klk X_r]$ with the
following properties:
\begin{itemize}
  \item $\mathcal{X}_r:=\{x_1\klk x_r\}$ is allowable for $f$ if and
  only if the equality $R_j(x_1\klk x_r)=0$ holds for $d-s\le j\le r-1$.
  \item Let $V\subset\Pp^r$ the projective closure of the affine
  variety of $\A^r$ defined by the polynomials $R_{d-s}\klk R_{r-1}$. Then $V$ is an
  ideal--theoretic complete intersection defined over $\fq$ of dimension
  $d-s$ and degree $s!/(d-r)!$, whose singular locus has codimension
  at least $3$.
  \item Let $V^0:=V\cap\{X_0=0\}\subset\Pp^{r-1}$. Then $V^0$ is an
  ideal--theoretic complete intersection defined over $\fq$ of dimension
  $d-s-1$ and degree $s!/(d-r)!$, whose singular locus has codimension
  at least $3$.
\end{itemize}
\end{theorem}

Combining Corollary \ref{coro: estimate codim 3} and Theorem
\ref{theorem: reduction value sets to counting} we obtain precise
information about the asymptotic behavior of $\chi(d,s,r)$ for
$d-s+1\le r\le d$, and thus of $N(d,s)$. More precisely, we have the
following result.
\begin{theorem}\label{theorem: counting for value sets}
Let assumptions and notations be as in Theorem \ref{theorem:
reduction value sets to counting} and set $D(s,d,r):=
\sum_{j=d-r+1}^{s}(j-1)$ and $\delta(s,d,r):=\prod_{j=d-r+1}^{s}j$.
We have the following estimate:
$$\left|\chi(d,s,r)-\frac{q^{d-s}}{r!}\right|\le \frac{15}{r!}
D(s,d,r)^3\delta(s,d,r)^2q^{d-s-1}.$$
\end{theorem}

From Theorem \ref{theorem: counting for value sets} we obtain the
following result concerning the behavior of $N(d,s)$.
\begin{corollary}\label{coro: average value sets}
With assumptions and notations as in Theorems \ref{theorem:
reduction value sets to counting} and \ref{theorem: counting for
value sets}, we have
$$|N(d,s)-\mu_d\, q|\le E(s,d):=\frac{e^{-1}}{2}+16\sum_{r=d-s+1}^{d}\frac{
D(s,d,r)^3\delta(s,d,r)^2}{r!}+\frac{2d}{q}.$$
A rough upper bound for the sum in the right--hand side of the
previous expression is $10d^{\,7}2^{-d}e^{2\sqrt{d}}$, which is
easily seen to tend to 0 as $d$ tends to infinity.
\end{corollary}
Corollary \ref{coro: average value sets} strengthens (\ref{eq:
average value set}) in several aspects. The first one is that our
result holds without any restriction on the characteristic $p$ of
$\fq$, while (\ref{eq: average value set}) holds for $p>d$. The
second aspect is that Corollary \ref{coro: average value sets} shows
that $N(d,s)=\mu_d\, q+\mathcal{O}(1)$, while (\ref{eq: average
value set}) only asserts that $N(d,s)=\mu_d\,
q+\mathcal{O}(q^{1/2})$. Finally, we obtain an explicit expression
for the constant underlying the $\mathcal{O}$--notation with a good
behavior, while (\ref{eq: average value set}) does not provide an
explicit expression for the corresponding constant. On the other
hand, it must be said that our result holds for $s\le d/2$, while
(\ref{eq: average value set}) holds without restrictions on $s$.
%
%


\begin{thebibliography}{10}

\bibitem{Ballico03}
E.~Ballico, \emph{An effective {Bertini} theorem over finite
fields}, Advances
  in Geometry \textbf{3} (2003), 361--363.

\bibitem{BaGiHeSaSc10}
B.~Bank, M.~Giusti, J.~Heintz, M.~Safey~El Din, and E.~Schost,
\emph{On the
  geometry of polar varieties}, Appl. Algebra Engrg. Comm. Comput. \textbf{21}
  (2010), no.~1, 33--83.

\bibitem{BaGiHeLePa12}
B.~Bank, M.~Giusti, J.~Heintz, L.~Lehmann, and L.M. Pardo,
\emph{Algorithms of
  intrinsic complexity for point searching in compact real singular
  hypersurfaces}, Found. Comput. Math. \textbf{12} (2012), no.~1, 75--122.

\bibitem{BaGiHeMb97}
B.~Bank, M.~Giusti, J.~Heintz, and G.M. Mbakop, \emph{Polar
varieties and
  efficient real equation solving: The hypersurface case}, J. Complexity
  \textbf{13} (1997), no.~1, 5--27.

\bibitem{BaGiHeMb01}
\bysame, \emph{Polar varieties and efficient real elimination},
Math. Z.
  \textbf{238} (2001), no.~1, 115--144.

\bibitem{BaGiHePa05}
B.~Bank, M.~Giusti, J.~Heintz, and L.M. Pardo, \emph{Generalized
polar
  varieties: {Geometry} and algorithms}, J. Complexity \textbf{21} (2005),
  no.~4, 377--412.

\bibitem{BiSD59}
B.~Birch and H.~{Swinnerton-Dyer}, \emph{Note on a problem of
{Chowla}}, Acta
  Arith. \textbf{5} (1959), no.~4, 417--423.

\bibitem{BrVe88}
W.~Bruns and U.~Vetter, \emph{Determinantal rings}, Lecture Notes in
Math.,
  vol. 1327, Springer, Berlin Heidelberg New York, 1988.

\bibitem{CaMa07}
A.~Cafure and G.~Matera, \emph{An effective {Bertini} theorem and
the number of
  rational points of a normal complete intersection over a finite field}, Acta
  Arith. \textbf{130} (2007), no.~1, 19--35.

\bibitem{Cohen72}
S.~Cohen, \emph{Uniform distribution of polynomials over finite
fields}, J.
  Lond. Math. Soc. (2) \textbf{6} (1972), no.~1, 93--102.

\bibitem{Cohen73}
\bysame, \emph{The values of a polynomial over a finite field},
Glasg. Math. J.
  \textbf{14} (1973), no.~2, 205--208.

\bibitem{DaKrSo13}
C.~{D'Andrea}, T.~Krick, and M.~Sombra, \emph{Heights of varieties
in
  multiprojective spaces and arithmetic {Nullstellens\"atze}}, Ann. Sci. \'Ec.
  Norm. Sup\'er. (4) \textbf{46} (2013), no.~4, 571--649.

\bibitem{Deligne74}
P.~Deligne, \emph{La conjecture de {Weil}. {I}}, Inst. Hautes
\'Etudes Sci.
  Publ. Math. \textbf{43} (1974), 273--307.

\bibitem{EdLiXi09}
F.~Edoukou, S.~Ling, and C.~Xing, \emph{Intersection of two quadrics
with no
  common hyperplane in $\mathbb{P}^n(\mathbb{F}_{\hskip-0.7mm q})$}, Preprint
  {\tt arXiv:0907.4556v1 [math.CO]}, 2009.

\bibitem{Fulton84}
W.~Fulton, \emph{Intersection theory}, Springer, Berlin Heidelberg
New York,
  1984.

\bibitem{GhLa02a}
S.~Ghorpade and G.~Lachaud, \emph{{\'Etale} cohomology, {Lefschetz}
theorems
  and number of points of singular varieties over finite fields}, Mosc. Math.
  J. \textbf{2} (2002), no.~3, 589--631.

\bibitem{GhLa02}
\bysame, \emph{Number of solutions of equations over finite fields
and a
  conjecture of {Lang} and {Weil}}, Number Theory and Discrete Mathematics
  (Chandigarh, 2000) (New Delhi) (A.K.~Agarwal et~al., ed.), Hindustan Book
  Agency, 2002, pp.~269--291.

\bibitem{Harris92}
J.~Harris, \emph{Algebraic geometry: a first course}, Grad. Texts in
Math.,
  vol. 133, Springer, New York Berlin Heidelberg, 1992.

\bibitem{Heintz83}
J.~Heintz, \emph{{Definability} and fast quantifier elimination in
  algebraically closed fields}, Theoret. Comput. Sci. \textbf{24} (1983),
  no.~3, 239--277.

\bibitem{HoPe68a}
W.~Hodge and D.~Pedoe, \emph{Methods of algebraic geometry. {Vol.}
{I}},
  Cambridge Math. Lib., Cambridge Univ. Press, Cambridge, 1968.

\bibitem{HoPe68b}
\bysame, \emph{Methods of algebraic geometry. {Vol.} {II}},
Cambridge Math.
  Lib., Cambridge Univ. Press, Cambridge, 1968.

\bibitem{Hooley91}
C.~Hooley, \emph{On the number of points on a complete intersection
over a
  finite field}, J. Number Theory \textbf{38} (1991), no.~3, 338--358.

\bibitem{Kleiman74}
S.~Kleiman, \emph{The transversality of a general translate},
Compos. Math.
  \textbf{28} (1974), no.~2, 287--297.

\bibitem{Kleiman76}
\bysame, \emph{The enumerative theory of singularities}, Real and
Complex
  Singularities, Oslo 1976: Proceedings of the 9th Nordic Summer School/NAVF
  Symposium in Mathematics, Oslo, Aug. 5--25, 1976 (P.~Holm, ed.), Sijthoff \&
  Noordhoff, 1976, pp.~297--396.

\bibitem{Kunz85}
E.~Kunz, \emph{Introduction to commutative algebra and algebraic
geometry},
  Birkh{\"a}user, Boston, 1985.

\bibitem{LeSc73}
D.~Lewis and S.~Schuur, \emph{Varieties of small degree over finite
fields}, J.
  Reine Angew. Math. \textbf{262/263} (1973), 293–--306.

\bibitem{LiNi83}
R.~Lidl and H.~Niederreiter, \emph{Finite fields}, Addison--Wesley,
Reading,
  Massachusetts, 1983.

\bibitem{Piene78}
R.~Piene, \emph{Polar classes of singular varieties}, Ann. Scient.
\'Ec. Norm.
  Sup. S\'er. 4 \textbf{11} (1978), no.~2, 247--276.

\bibitem{Samuel67}
P.~Samuel, \emph{M\'ethodes d'alg\`ebre abstraite en g\'eom\'etrie
  alg\'ebrique}, Springer, Berlin Heidelberg New York, 1967.

\bibitem{Schmidt76}
W.~Schmidt, \emph{Equations over finite fields. {An} elementary
approach},
  Lectures Notes in Math., no. 536, Springer, New York, 1976.

\bibitem{Serre91}
J-P. Serre, \emph{Lettre \`a {M}. {T}sfasman}, Ast\'erisque
\textbf{198-200}
  (1991), 351--353.

\bibitem{Shafarevich94}
I.R. Shafarevich, \emph{Basic algebraic geometry: {Varieties} in
projective
  space}, Springer, Berlin Heidelberg New York, 1994.

\bibitem{Teissier82}
B.~Teissier, \emph{Vari\'et\'es polaires. {II}: {Multiplicit\'es}
polaires,
  sections planes et conditions de {Whitney}}, Algebraic geometry, Proc. Int.
  Conf., La R\'abida/Spain 1981 (Berlin Heidelberg New York) (J.~Aroca,
  R.~Buchweitz, M.~Giusti, and M.~Merle, eds.), Lect. Notes Math., vol. 961,
  Springer, 1982, pp.~314--491.

\bibitem{Teissier88}
\bysame, \emph{Quelques points de l'histoire des vari\'et\'es
polaires, de
  {Poncelet} \`a nos jours}, S\'eminaire d'analyse: 1987--1988
  (Clermont--Ferrand), vol.~4, Univ. Blaise--Pascal, 1988.

\bibitem{Uchiyama55b}
S.~Uchiyama, \emph{Note on the mean value of {$V(f)$}. {II}}, Proc.
Japan Acad.
  \textbf{31} (1955), no.~6, 321--323.

\bibitem{Vogel84}
W.~Vogel, \emph{Results on {B\'ezout}'s theorem}, Tata Inst. Fundam.
Res. Lect.
  Math., vol.~74, Tata Inst. Fund. Res., Bombay, 1984.

\bibitem{Wooley08}
T.~Wooley, \emph{Artin's {Conjecture} for septic and unidecic
forms}, Acta
  Arith. \textbf{133} (2008), no.~1, 25--35.

\bibitem{Zahid10}
J.~Zahid, \emph{Nonsingular points on hypersurfaces over $\fq$}, J.
Math. Sci.
  (N. Y.) \textbf{171} (2010), no.~6, 731--735.

\end{thebibliography}
\providecommand{\bysame}{\leavevmode\hbox
to3em{\hrulefill}\thinspace}
\providecommand{\MR}{\relax\ifhmode\unskip\space\fi MR }
\providecommand{\MRhref}[2]{%
  \href{http://www.ams.org/mathscinet-getitem?mr=#1}{#2}
} \providecommand{\href}[2]{#2}

\end{document}